\documentclass[12pt,a4paper,oneside]{amsart}
\usepackage{amsmath,amsfonts,amssymb,amsthm,amscd}
\usepackage[english]{babel}
\usepackage{graphicx,epic,multicol}
\usepackage[T1]{fontenc}
\usepackage{esint}
\usepackage[colorlinks=true]{hyperref}
\usepackage{pgf, tikz}
\usepackage{color}
\usepackage{dsfont}

\input{comment.sty}
\includecomment{MM}
\includecomment{CL}

\renewcommand{\a}{\alpha}

\newcommand{\g}{\gamma}
\newcommand{\G}{\Gamma}
\renewcommand{\d}{\delta}

\renewcommand{\k}{\kappa}

\renewcommand{\l}{\lambda}

\newcommand{\m}{\mu}

\newcommand{\e}{\varepsilon}
\newcommand{\f}{\varphi}

\newcommand{\N}{{\mathbb N}}

\newcommand{\Z}{{\mathbb Z}}

\newcommand{\comp}{\mathsf{c}}

\newcommand{\ie}[0]{\emph{i.e.} }

\newcommand{\diam}[0]{\mathrm{Diam}}
\newcommand{\del}{\partial}

\newcommand{\rd}[0]{\mathrm{rd}}
\newcommand{\wid}[0]{\mathrm{wd}}

\newtheorem*{claimn}{Claim}
\newtheorem{theorem}{Theorem}[section]
\newtheorem{proposition}[theorem]{Proposition}
\newtheorem{lemma}[theorem]{Lemma}
\newtheorem{corollary}[theorem]{Corollary}

\theoremstyle{definition}
\newtheorem{definition}[theorem]{Definition}

\newtheorem{notation}[theorem]{Notation}

\theoremstyle{remark}
\newtheorem{remark}[theorem]{Remark}

\newtheorem{question}[theorem]{Question}
\newtheorem{conjecture}[theorem]{Conjecture}

\begin{document}

\title[Connectedness of spheres]
      {Connectedness of spheres in Cayley graphs}
      \author{J\'{e}r\'{e}mie Brieussel \& Antoine Gournay}

\begin{abstract}
We introduce the notion of connection thickness of spheres in a Cayley graph, related to dead-ends and their retreat depth. It was well-known that connection thickness is bounded for finitely presented one-ended groups. We compute that for natural generating sets of lamplighter groups on a line or on a tree, connection thickness is linear or logarithmic respectively. We show that it depends strongly on the generating set. We give an example where the metric induced at the (finite) thickness of connection gives diameter of order $n^2$ to the sphere of radius $n$. We also discuss the rarity of dead-ends and the relationships of connection thickness with cut sets  in percolation theory and with almost-convexity.  Finally, we present a list of open questions about spheres in Cayley graphs.
\end{abstract}
\maketitle

\section{Introduction}

Spheres are simply beautiful. Long ago Plato considered the world was given the shape of a (euclidean) sphere for it is the most perfect (symmetric) of all, Timaeus 34b \cite{Pla}. Nowadays the most popular game in the world consists in two teams playing with a sphere and not allowed to touch it by the hands. As mathematicians we also love to play with spheres, and as geometric group theorists we shall focus on spheres in Cayley graphs of infinite groups.

Spheres have not received much attention for themselves in geometric group theory, at the notable exception of a study by Duchin, Leli\`evre and Mooney \cite{DLMfag}, \cite{DLM}, \cite{DLM3}, somewhat focused on the Abelian case. The reason is probably that spheres are often either too easy or too complicated to describe. For instance the Cayley graph structure restricted to a sphere gives a graph with no edges for such elementary examples as free Abelian groups. To bypass this, we rather consider thickenings of spheres, and denote $S(n,r)$ the subgraph obtained by keeping vertices at distance from identity between $n$ and $n+r$ as well as edges between them.

The first natural question is that of connectedness: does there exist a thickness $r$ such that the graphs $S(n,r)$ are connected for all $n$ ? A first obstruction is given by the topological notion of ends. Recall that a graph has at most $k$ ends if the complement of any finite set has at most $k$ connected components. By a result of Stallings \cite{Sta}, the number of ends (the least such $k$) of a Cayley graph of an infinite group is either $2$, for virtually cyclic groups, or $\infty$, for free products with amalgamation and HNN-extensions over finite groups, or $1$ for any other group. For instance in a free group with free generating set the number of connected components of $S(n,r)$ is always the number of elements in the sphere of radius $n$. But as groups in the two first cases of Stallings' classification are well-understood, we may focus on the generic case of one-ended groups.

The geometric notion of dead-ends gives a second natural obstruction to spheres, or rather their finite thickenings, being connected. As defined by Bogopol'skii \cite{Bog}, a vertex of a Cayley graph is called a dead-end if it is not adjacent to a vertex further away from the identity. A dead-end $g$ in the sphere of radius $n$ can be measured by two different means : its width (often called depth in the litterature) is the distance between $g$ and the infinite component of the complement of the ball $B(n-1)$ and its retreat depth (also sometimes called depth) is the least $d\geq 0$ such that $g$ belongs to the infinite component of the complement of $B(n-d-1)$. There are examples of finitely generated groups with dead-ends of arbitrarily large width and retreat-depth (see \cite{CT}, \cite{CR}, \cite{War-Heis} and \S{}\ref{ssdef}). If the Cayley graph contains a dead-end on the sphere of radius $n+r$ of retreat-depth $d \geq r$, then the complement of the ball $B(n-1)$ is not connected. So the thickened sphere $S(n,r)$ is not connected either.

One naturally wonders if dead-ends give the only obstruction to connectedness of (thickened) spheres for one-ended groups. Let $S(n,r)^\infty$ be the sphere of radius $n$ and thickness $r$ where the dead-end components have been removed. Namely, this is the intersection of $S(n,r)$ with the only infinite connected component of the complement of $B(n-1)$. It is easily checked (see \S{}\ref{ssdef}) that for each $n$ there exists an integer $r$ such that $S(n,r)^\infty$ is connected. This justifies the following 

\begin{definition}\label{defth}
Let $G$ be a group together with a finite generating set $S$. We define the {\it connection thickness} $\textrm{th}_{G,S}(n)$ of the $n^\textrm{th}$ sphere of $(G,S)$ to be the minimal $r$ such that $S(n,r)^\infty$ is connected. 
\end{definition}
We will be interested in the connection thickness function $\textrm{th}_{G,S}: \Z_{\geq 0} \to \Z_{\geq 0} \cup \{ \infty \}$.
When $G$ has more than one end, $\textrm{th}_{G,S}(n)$ is infinite for large enough $n$. When $G$ is one-ended, $\textrm{th}_{G,S}(n)$ is necessarily finite for each $n$.  When the connection thickness is bounded, we say that the Cayley graph of $(G,S)$ has {\it connected spheres}.

For finitely presented groups, it is known that spheres are connected.
\begin{theorem}\label{casfinpres}
Let $G$ be a finitely presented group, $R$ a set of words generating all relations, and $\ell = \tfrac{1}{2} \max_{w \in R} |w|$. Then $S(n,r)^\infty$ is connected for any $r \geq \lfloor \ell \rfloor$, i.e. the connection thickness is bounded above by $\lfloor\ell\rfloor$.
\end{theorem}
This is the content of the small note \cite{Gou}. However, R.~Lyons pointed out to the second author that the above theorem also follows from results of Babson \& Benjamini \cite{BB}. See also Tim{\'a}r \cite[Theorem 5.1]{Timar} and the book by Lyons and Peres \cite[Lemma 7.28]{Lyons} for slightly different arguments (on more general graphs).

For some one-ended but not finitely presented groups, spheres are not connected. More precisely, we prove

\begin{theorem}\label{nonconnected-intro} Let $L$ be a finite group.
\begin{enumerate}
\item For the lamplighter group $\Z \wr L$ on a line with respect to the ``switch-walk-switch'' generating set $L\{\pm1\} L$,the connection thickness is  \[
\mathrm{th}_{\Z \wr L,L \{\pm1\} L}(n)=n+2.
\]
Moreover, the set $S(n,r)^\infty$ has at least $2|L|^{n-r}$ connected components while $1 \leq r+1 \leq n$.
\item Let $T_d$ be a group with Cayley graph the $d$-regular tree. For the lamplighter group $T_d \wr L$ on a tree with ``switch-walk-switch'' generating set $LS_TL$, there are constants $c>0$ and $K>1$ such that the connection thickness satisfies 
\[
|\mathrm{th}_{T_d \wr L,L S_{T_d} L}(n)-\log_{d-1}(n)| \leq c,
\]
and for any $r \leq \log_{d-1}(n)-c$ the number of connected components in $S(n,r)^\infty$ is at least $K^{n}$.
\end{enumerate} 
\end{theorem}

For the lamplighter on a line, the non-connectedness is maximal in the following quantitative sense: the entropy of the partition into connected components with respect to the counting measure is asymptotically (in $n$) maximal while the thickness $r$ is less than $\frac{n}{8}$, see Proposition \ref{entmax} for a precise statement. This implies in particular that there are no ``gigantic'' connected components in the spheres of $\Z \wr L$.

It is not surprising that lamplighter groups provide interesting examples of connection thickness, as they already provided natural examples of Cayley graphs with dead-ends of unbounded retreat depth \cite{CT}. More generally, they provide very interesting examples of groups, see among many others \cite{KV}, \cite{LPP} \cite{GZ}, \cite{Ers}.

Note that by \cite[Proposition 5]{FGO} (see also \S{}\ref{sec:rarde} below), the retreat depth of an element $x$ is at most $|x|_S/2$. This implies that dead-ends can prevent connectedness on a thickness at most $n$. So it would be tempting to believe there is a uniform linear upper bound $\textrm{th}_{G,S} (n) \leq n$. However, Theorem \ref{nonconnected-intro}.(1) shows this is not the case. We do not know examples of one-ended Cayley graphs with connection thickness bigger than $n+2$, see Question \ref{thgrowth}.

Both examples in Theorem \ref{nonconnected-intro}  are groups with dead-ends of arbitrarily large retreat depth. This raises questions on how connected spheres and retreat depth are related. 

It is particularly easy to estimate connection thickness and the retreat depth for direct products $G_1 \times G_2$, where $G_1$ and $G_2$ are infinite groups with respective generating sets $S_1$ and $S_2$. 
For the ``product'' generating set $S_\vee=\left(S_1\cup\{e_1\}\right)\times \left(S_2\cup\{e_2\}\right)$, one has $|(g_1,g_2)|_{S_\vee}=\max\left(|g_1|,|g_2| \right)$ and both the connection thickness and the retreat depth vanish. For the ``summed'' generating set $S_\perp=\left(S_1\times\{e_2\}\right) \cup \left(\{e_1\}\times S_2\right)$, one has $|(g_1,g_2)|_{S_\perp}=|g_1|+|g_2|$, the connection thickness is $1$ and the retreat depth is the minimum of the retreat depth of $G_1$ and $G_2$ (see \S{} \ref{sdirprod}). Applied to the direct product of two copies of the lamplighter group on a line, this gives:

\begin{proposition}\label{prop-simple-intro}
Let $L$ be a finite group and $G = (\Z \wr L) \times (\Z \wr L)$. Then $G$ admits a generating set with unbounded retreat depth and connected spheres of thickness $1$ and $G$ admits another generating set with both retreat depth zero and connected spheres of thickness zero.
\end{proposition}

In particular, this gives a simple example of a group where the retreat depth of dead-ends varies greatly with the generating set (according to \cite{War}, this is already the case for $\Z \wr \Z_2$, but our example is simpler). Moreover, Proposition \ref{prop-simple-intro} gives an exemple of a group with dead-ends of arbitrarily large retreat depth and connected spheres. We know no example of a group with bounded retreat depth and not-connected spheres, i.e. unbounded connection thickness (Question~\ref{qnatur}).

It is also natural to wonder how much connection thickness may depend on the generating set for a given group. We show this dependence can be quite strong, answering Question (iii) in \cite{Gou}.

\begin{theorem}\label{teo-echelle-intro}
Let $L$ be a finite group and $G=\left( \Z \times \Z_2 \right) \wr L$. Then 
\begin{enumerate}
 \item the group $G$ admits a generating set with connection thickness $\leq 24$ (spheres are connected) and retreat depth of  dead-ends at most $5$.
 \item the group $G$ admits a generating set with connection thickness $\mathrm{th}_{G,S}(n)=n+2$ (spheres are not connected) and dead-ends of unbounded retreat depth. 
\end{enumerate}
\end{theorem}

With its natural generating set, the Cayley graph of $\Z \times \Z_2$ is a ladder. The first generating set of $G$ in Theorem \ref{teo-echelle-intro} is the associated ``switch-walk-switch'' generating set. The second generating set of $G$ may be described as ``switch the two lamps at the current $\Z$-coordinate-walk-switch the two lamps'', which makes the associated Cayley graph perfectly similar to that of a lamplighter group on a line as in Theorem \ref{nonconnected-intro}(1).

The property of connected boundaries of a Cayley graph (see \cite{BB}, \cite{Timar} and \S{}\ref{ssconnbnd}) is stronger than that of connected spheres. It is an open question whether connected boundaries are invariant under change of generating set or quasi-isometry. Theorem \ref{teo-echelle-intro} made the lamplighter group on a ladder a natural candidate for a negative answer. However we check in \S{}\ref{ssconnbnd} that it does not have connected boundaries for the mentionned generating sets.

When the spheres are connected, it is natural to study the distortion between the metric induced on $S(n)^\infty$ by the ambient Cayley graph with the metric induced on it by restriction to $S(n,r)^\infty$ and we may consider the distortion is infinite when the spheres are not connected. In this direction, we simply observe the following.

\begin{theorem}\label{ZwrZ-intro}
In the Cayley graph of the group $\Z \wr \Z$ with generating set ``walk or switch'', the graph $S(n,2)$ is connected and has diameter $\asymp n^2$. Precisely, there exists $c_1,c_2>0$ such that $c_1n^2 \leq \mathrm{diam} S(n,2) \leq c_2 n^2$.
\end{theorem}

The choice of $\Z$ is not important to show that 2-thickened spheres are connected. This holds for $\G\wr L$ whenever $\G$ and $L$ are infinite and $L$ has no dead-end, see Theorem \ref{distortion}. 

In contrast, the diameter of the sphere $S(n)$ with respect to the ambient group metric is necessarily $2n$ in an infinite group. This is related to the notion of sprawl of a sphere introduced in \cite{DLM}. The sprawl of a set is the average distance between two points chosen (uniformly) randomly and independently. The arguments that prove Theorem \ref{ZwrZ-intro} can be extended to show that the sprawl of the graph $S(n,2)$ is also $\asymp n^2$.

On the other hand, the sprawl of spheres with respect to the induced metric is $\asymp n$ for Abelian groups by \cite{DLM}, and at least exponential for hyperbolic groups by \cite{Ger} as explained in Remark \ref{disthyp}. This seems to indicate that for groups of exponential growth the notion of sprawl with respect to the induced metric is finer than with respect to the ambient metric.

{\bf Organisation of the paper.} 
 Precise definitions and notations, as well as elementary observations, are  given in \S{}\ref{sectiondef}, split between \S{}\ref{ssdef} about ends, dead-ends, spheres and connection thickness and \S{}\ref{sec:ll} about lamplighter groups.
The first part of Theorem \ref{nonconnected-intro} is established  in \S{}\ref{llline}, by concatenation of Proposition \ref{llcomp} and Theorem \ref{llconn}. We also establish estimates on the number of dead-ends in $\Z \wr L$ in Proposition \ref{relsize}, and on the entropy of the partition into connected components of $S(n,r)^\infty$ in Proposition~\ref{entmax}.
The second part of Theorem \ref{nonconnected-intro} is the object of \S{}\ref{slltree}. It follows from Corollary \ref{cornoncon} and Proposition \ref{treeconnection}. 
Theorem \ref{teo-echelle-intro} is derived in \S{}\ref{sec:ladder}. The main step takes the form of Theorem \ref{techelle}. \S{}\ref{sec:distortion} is devoted to Theorem \ref{ZwrZ-intro}.  A lot of elementary observations about dead-ends and connected spheres are gathered in \S{} \ref{sgl}, about direct products of groups in \S{}\ref{sdirprod}, where we derive Proposition \ref{prop-simple-intro}, and about rarity of dead-ends in groups in \S{}\ref{sec:rarde}. The relationship of connectedness of spheres with well-known topics in geometric group theory is discussed in \S{}\ref{other}. The property of connected boundaries is recalled in \S{}\ref{ssconnbnd}, and the relationship with almost-convexity is explained in \S{}\ref{ssalmcvx}. Finally, many open questions are presented in \S{}\ref{qu}.

Formally, all sections from \S{}\ref{llline} to \S{}\ref{other} can be read independently. However we recommend to read \S{}\ref{sec:description} and \S{}\ref{sec:nonconnline} before \S{}\ref{slltree} and \S{}\ref{sec:ladder}.

{\it Acknowledgments:} Antoine Gournay is supported by the ERC-StG 277728 ``GeomAnGroup''.

\section{Definitions and notations}\label{sectiondef}

\subsection{Ends, spheres  and dead-ends}\label{ssdef}

We are mostly interested in a group $G$ with a generating set $S$ by studying its Cayley graph $Cay(G,S)$. This enables us to endow the group with the word metric $d$, or equivalently, the combinatorial graph distance. 

{\bf Ends.} 
Finite graphs have no end. Let $\mathcal{P}_{\mathrm{fin}}(X)$ (resp. $\mathcal{P}_{\mathrm{inf}}(X)$) denote the finite (resp. infinite) subsets of $X$.
\begin{definition}
An {\bf end} of an infinite connected graph with vertex set $X$ is a function $\xi:\mathcal{P}_{\mathrm{fin}}(X) \to \mathcal{P}_{\mathrm{inf}}(X)$ so that, for any $F$ and $F'$, $\xi(F)$ is an infinite connected component and $\xi(F) \cap \xi(F') \neq \emptyset$. Equivalently, an end is a coset of the space of infinite {\it simple} rays in the graph where two rays are equivalent whenever there is another ray that contains infinitely many points of both of them.
\end{definition}
It follows that a graph $\mathcal G$ has {\bf at most $k$ ends} if for any finite subset $F \subset X$, the subgraph $\mathcal G \setminus F$ has at most $k$ infinite connected components. Therefore, the {\bf number of ends} is defined as follows. Take an increasing and exhausting sequence of finite sets $F_n$, and let $k_n$ to be the number of  infinite connected component of $F_n^\comp$. Then the number of ends is $\lim_n k_n \in [1,\infty]$.

A result of Stallings asserts that the number of ends of a Cayley graph of a group is either $0$, for finite groups, or $2$, for virtually cyclic groups, or $\infty$, for free products with amalgamation or HNN-extensions over finite groups, or $1$ for any other group \cite{Sta}.

{\bf Spheres.} In the Cayley graph of a group $G$ with respect to a generating set $S$, we denote $e$ the neutral element of $G$ and $d$ the graph distance.
\begin{notation}
Let $n,r\in \Z_{\geq 0}$.  Denote 
\[
\begin{array}{lll}
B(n) &= \{g\in G|d(g,e)\leq n\}  & \text{the {\bf ball} of radius } n,\\
S(n) & =\{g\in G|d(g,e)=n\}      & \text{the {\bf sphere} of radius} n,\\ 
S(n,r) &= B(n+r)\setminus B(n-1) & \text{the {\bf annulus} of radius $n$ and thickness } r, 
\end{array}
\]
When $G$ is one-ended, let $B(n)^{c,\infty}$ denote the infinite component of the complement of the ball of radius $n$, then 
\[
S(n,r)^\infty=B(n+r)\cap B(n-1)^{c,\infty}
\]
are the elements of the annulus connected to infinity. 
We will also often use the shorthand $S(n)^\infty:=S(n,0)^\infty$.
\end{notation}

For a group with infinitely many ends, the number of connected components of $S(n,r)$ obviously tends to infinity with $n$ for any $r \geq0$ (for instance $S(n,r)$ has at least $d(d-1)^{n-1}$ connected components in a $d$-regular tree). This justifies that we focus our study on groups with one end.

As mentioned in the introduction, there is another obstruction to the connectedness of $S(n,r)$: dead-ends.

{\bf Dead-ends.} 
Recall that as defined by Bogopol'skii \cite{Bog}, a vertex of a Cayley graph is called a dead-end if it is not adjacent to a vertex further away from the identity. Two notions of ``depth'' of these dead-ends have been introduced. The most commonly seen in the literature will be here called the width (although it usually bears the name ``depth'') and the other is the retreat depth.
\begin{definition}
The {\bf width} of a dead-end element $g$, denoted $\wid(g)$, is the distance between $g$ and $B(|g|)^{c,\infty}$. 

The {\bf retreat depth} of an element $g$, denoted $\rd(g)$ is the least $d \geq 0$ such that $g$ belongs to an infinite component of $B(|g|-d-1)^c$. 
\end{definition}
Occasionally, we will write $\rd(G,S)$ for $\sup_{g \in G} \rd(g)$ (the generating set being $S$). Likewise for $\wid(G)$.

Bogopol'skii \cite{Bog} showed that the width is always bounded by a constant in a hyperbolic group. Warshall \cite{War-Heis} has shown that the Heisenberg groups has ``large'' dead-ends which are ``shallow'', i.e. the width may be arbitrarily big but the retreat depth is at most $2$. Earlier examples go back to Cleary and Taback: they have shown that lamplighter groups on $\Z$ with finite lamps have (many) dead-ends of arbitrary large width \cite{CT}, a result also known to hold for some finitely-presented groups by Cleary and Riley \cite{CR}. Lastly, \cite{RW} have shown that having dead-ends of bounded width is not an invariant of the generating set.

According to Definition \ref{defth},
we say the group has {\bf connected spheres} if there is an integer $r \geq 0$ so that for all $n \in  \Z_{\geq 0}$, $S(n,r)^\infty$ is connected. 

The {\bf connection thickness} of $G$ with respect to $S$ is the function $\textrm{th}_{G,S}: \Z_{\geq 0} \to \Z_{\geq 0}\cup\{\infty\}$ defined by 
\[
\textrm{th}_{G,S}(n) := \min \{r \in \Z_{\geq 0} \mid S(n,r)^\infty \text{ is connected}\}. 
\]

When $G$ has more than one end, $\textrm{th}_{G,S}(n)$ is infinite for large enough $n$. On the other hand, when $G$ is one-ended, $\textrm{th}_{G,S}(n)$ is necessarily finite for each $n$: this is a consequence of the following lemma.
\begin{lemma}\label{monot}
Assume $0 \leq \ell \leq k $ and $n \in \Z_{\geq 0}$. Then $S(n,k)^\infty$ is connected if and only if any two elements of $S(n+\ell) \cap B(n-1)^{\comp,\infty}$ are connected by paths staying inside $S(n,k)^\infty$.
\end{lemma}
\begin{proof}
The proof consists in noticing that any element of $S(n,k)^\infty$ is connected to an element of $S(n+\ell)$ by a path staying inside $S(n,k)^\infty$. Assume $x \in S(n,k)^\infty$. If $|x| = n+\ell$ there is nothing to prove.

If $n+k \geq |x|> n+\ell$, consider a geodesic path from $x$ to the identity. This path crosses $S(n+\ell)$ at an element $y$ and, between $x$ and $y$, stays inside $S(n,k)^\infty$. This implies $y \in S(n+\ell) \cap B(n-1)^{\comp,\infty}$. 

If $n \leq |x|< n+\ell$, consider a path from $x$ to infinity. This path crosses $S(n+\ell)$ at an element $y$. Taking $y$ minimal, the path stays inside $S(n,k)^\infty$ between $x$ and $y$.
\end{proof}

\subsection{Lamplighter groups}\label{sec:ll}

A group $G$ is called a lamplighter group, or a wreath product, if it has the form of a semi-direct product $G=\G \wr L=\G \ltimes (\Sigma_\G L)$, between a base group $\G$ and the finitely supported functions from $\G$ to a lamp group $L$. Its elements have the form $g=(\g, f)$, where $f:\G \rightarrow L$ is the lamp function with finite support. Let $S_\G,S_L$ be generating sets of the groups $\G,L$. By abuse of notation, we still denote $S_\G=\{(\g,Id)|\g \in S_\G\}$ and $S_L=\{(e_\G,\d_s)|s \in S_L\}$ the subsets of $G= \G \wr L$, where $e_\G$, $e_L$ denote the neutral elements of the groups $\G$ and $L$ respectively, $Id(x)=e_L$ for all $x \in \G$, $\d_s(x)=e_L$ for all $x \neq e_\G$ and $\d_s(e_\G)=s$. The set $S_\G \cup S_L$ is the canonical generating set of $G$, also called the ``switch or walk'' generating set. 

The name ``lamplighter group'' comes from the following interpretation. An element $g=(\g, f)$ of the group is described by a configuration of lamps taking values in $L$ on the base group $\G$ (this configuration is given by the function $f$), together with a position $\g$ of a ``lamplighter'' in the base group $\G$.

The action of a generator $\g'$ in $S_\G$ has the form $\left(\g, f\right)(\g',Id)=\left(\g\g', f\right)$
so can be interpreted as a walk of the lamplighter from position $\g$ to $\g\g'$. The action of a generator $s=(e_\G,\d_s)$ in $S_L$ has the form $ \left(\g, f\right)(e_\G,\d_s)=\left(\g,f'\right)$ where $f'(x)=f(x)$ for all $x \neq \g$ and $f'(\g)=f(\g)s$
so can be interpreted as a switch of the lamp in position $\g$, the position of the lamplighter, from intensity $f(\g)$ to $f(\g)s$. 

As both $S_\G$ and $S_L$ contain the identity $e$, the set  $S_LS_\G S_L$ is also generating, called the ``switch-walk-switch'' generating set. When $S_L =L$, this generating set is especially interesting as the word metric is computed by a travelling salesman problem in the Cayley graph of $\G$ with respect to $S_\G$.

\begin{definition}\label{defTSP}
Let $x,y$ be elements of a graph and $A$ a subset of vertices. The travelling salesman distance from $x$ to $y$ through $A$ noted $d_{TS}(x,A,y)$ is the length of the shortest path starting at $x$ and ending at $y$ which passes at least once through all the vertices of~$A$. 
\end{definition}

The distance in the lamplighter groups (or graphs) is given by this length. 

\begin{proposition}\label{propTSP}
Assume $S_L =L$ and endow the group $\G\wr L$ with the distance associated to the ``switch-walk-switch'' generating set $LS_{\Gamma}L$.
Let $g=(\g,f) \in \Gamma \wr L$. Assume $g$ is not of the form $(e_\G,  \delta_{s})$ for some $s\in L$ and denote $A = \{ y \in \G \mid f(y) \neq e_L\}$ (the set of lamps which are on), then $$|g| = d_{TS}(e_\G,A,\g).$$  
\end{proposition}

In other words, if $\Gamma \wr L \twoheadrightarrow \Gamma$ is the natural projection, then the word length of $g = (\g,f) \in \G \wr L$ is the length of the shortest path in the Cayley graph of $\G$ for $S_\G$ which starts at $e_\G$, covers all elements in the support of $f$ and ends at $\g$.

Observe that this generating set is finite if and only if $L$ is finite, and that the length of an element of the form $(e_\G,  \delta_{s})$ is obviously $1$.

\begin{proof}
Let $w$ be a representative word of $g$. Denote $w_\G$ the word obtain by forgetting all generators of $S_L$ in $w$. The length of $w_\G$ is at least $d_{TS}(e_\G,A,\g)$ because only lamps at the sites in $\G$ visited by the path described by $w_\G$ can be switched on. This gives the lower bound.

To get the upper bound, consider $w_\G$ a word in $S_\G$ describing a solution to the associated travelling salesman problem. Then we can extend $w_\G$ to a word of the same length in $S_LS_\G S_L$ representing $g$. Indeed, at each step, as $S_L=L$ we can set the lamps at departure and arrival of the lighter to any chosen value in $L$.
\end{proof}

\section{The lamplighter on a line with finite lamps}\label{llline}

The aim of this section is to establish  Theorem \ref{nonconnected-intro}.(1).
Further, we establish in Proposition \ref{relsize} an upper bound on the number of dead-end elements in the group $\Z \wr L$, and show that the entropy of the partition into connected components of $S(n,r)^\infty$ is asymptotically maximal in Proposition~\ref{entmax}.

\subsection{Description of elements and their length} \label{sec:description}

Let $L$ be a finite group and consider the group $G=\Z \wr L$ together with its ``switch-walk-switch'' generating set $L\{\pm 1\} L$. An element of $G$ is described by a lamp function $f:\Z \rightarrow L$ and the position $z$ in $\Z$ of the lamplighter. We write $g=(z,f)$.

To an element $g$ are associated $a(g)=\min(\{t\in \Z|f(t)\neq e\}\cup \{0,z\})\leq0$ and $b(g)=\max(\{t\in \Z|f(t)\neq e\}\cup \{0,z\})\geq0$, where $a(g)$ (resp. $b(g)$) is the minimal (resp. maximal) lamp turned on or the position of the lamplighter or $0$. We have $a(g)=0$ (resp. $b(g)=0$) only if all negative (resp. positive) lamps are turned off. For simplicity, we write $a$ instead of $a(g)$, $a_k$ for $a(g_k)$ and so on.

\begin{lemma}
Let $L$ be a finite group. Consider the group $\Z \wr L$ endowed with ``switch-walk-switch'' generating set $L\{\pm 1\} L$.
Let $g=(z,f)$ be an element of $\Z \wr L$. If $g \notin S_L$ then the word length is $|g|=2b+2|a|-|z|$.
\end{lemma}

\begin{proof}
Let $u$ denote the generator $+1$ of $\Z$ with multiplicative notation. We have $g=u^{a}f(a)uf(a+1)u\dots uf(b)u^{b-z}$, word of (switch-walk-switch) length  $2b+2|a|-z$, interpreted as ``go to $a$ without switching the lamps, then go to $b$ switching the lamps appropriately, then go back to $z$ without switching the lamps''. Similarly, we have $g=u^bf(b)u^{-1}f(b-1)u^{-1}\dots u^{-1}f(a)u^{-a+z}$, of length $2b+2|a|+z$. One of them has to give a solution to the travelling salesman problem of Proposition \ref{propTSP}, since the lamplighter must start from $0$ and end his walk at $z$ going meanwhile through $a$ and $b$.
\end{proof}

\subsection{Non-connectedness of spheres}\label{sec:nonconnline}

\begin{lemma}\label{connectedcomp}
Let $n,r \in \Z_{\geq 0}$ and consider two elements $g,g'$ in $S(n)$ with $z \geq 0$.
\begin{enumerate}
\item If $z-r>0$, then $g$ and $g'$ are in the same connected component of $S(n,r)$ if and only if $a'=a$, $b'=b$, $z'=z$ and $\forall t \notin [z-r,z]$, $f(t)=f'(t)$.
\item If $z-r\leq 0$ and $z<\min(b,|a|)$, then $g$ and $g'$ are in the same connected component of $S(n,r)$ if and only if $a'=a$, $b'=b$, $|z'|=z$ and $\forall t \notin [-z,z]$, $f(t)=f'(t)$. 
\end{enumerate}
\end{lemma}

Of course a similar statement holds for $z \leq 0$. We could probably also give a complete description of the connected component of $g$ when $-z<a$, but it would be more complicated.

\begin{proof}
We present a detailed proof of case (1) when $z-r>0$. First assume $z=z'$ and $\forall t \notin [z-r,z]$, $f(t)=f'(t)$. Observe that:
$$g'=gu^{-r}f(z-r)^{-1}f'(z-r)uf(z-r+1)^{-1}f'(z-r+1)\dots uf(z)^{-1}f'(z). $$
This equality provides a path $g_k=gu^{-k}$ of norm $n+k$ for $0 \leq k \leq r$ and $g_k=gu^{-r}f(z-r)^{-1}f'(z-r)\dots uf(z-2r+k)^{-1}f'(z-2r+k)$ of norm $n+r-k$ for $r \leq k \leq 2r$, with $g_0=g$ and $g_{2r}=g'$. This path is interpreted as ``the lamplighter moves left to position $z-r$ without switching lamps, then he moves back to position $z$ switching on lamps appropriately''.

For the converse implication, we first treat the case $z<b$. Assume by contradiction that there exists $t \notin [z-r,z]$ with $f(t)\neq f'(t)$, and a path $g=g_0,g_1,\dots ,g_K=g'$. Let $k$ be minimal such that $z_k=z+1$ or $z_k=z-r-1$. By minimality and since $0<z-r\leq z<b$, we must have $a_k=a$ and $b_k=b$, so $|g_k|=n-1$ or $|g_k|=n+r+1$, raising a contradiction. This shows that $\forall t \notin [z-r,z]$, $f(t)=f'(t)$, hence $a'=a$ and $b'=b$. By equality of norms of $g$ and $g'$, this forces $z=z'$. We interpret this by saying that ``the lamplighter cannot move right of position $z$ or left of position $z-r$ without exiting the annulus $S(n,r)$''.

Now treat the case $z=b$. We first prove that $\forall t \leq z-r, f(t)=f'(t)$. Otherwise, there would be a path $g=g_0,g_1,\dots,g_K=g'$ in $S(n,r)$, and a $k$ minimal such that $z_k=z-r-1$. By minimality, $a_k=a$, which forces $b_k \leq b-1$. In particular, there exists a $k'$ minimal with $b_{k'} \leq b-1$. This forces $a_{k'}=a$ and $b_{k'}=z_{k'}=b-1$, so $|g_{k'}|=n-1$, contradiction. There remains to see that $z'=b'=b$. By norm, we must have $\g'\leq b' \leq b$. If we had $b'<b$, then as above there would exist a $k$ minimal with $b_k \leq b-1$, raising a contradiction. Thus $b'=b$ and by norm $z'=b'$. As an interpretation, ``staying in $S(n,r)$ forces the lamplighter to have a position in $[z-r,z+r]$, with the rightmost non-trivial lamp at position $b$ if the lamplighter is in $[z-r,z-1]$''.

In case (2) we have $a<-z$ and $z<b$, the lamplighter starts in position $z$. He cannot move right unless entering $S(n-1)$. Each step left increases the norm by one until  $z_{k_1}=0$, with $g_{k_1} \in S(n+z)$. Afterwards, each step left decreases the norm by one until $z_{k_2}=-z$, with $g_{k_2} \in S(n)$. At this position, the lamplighter cannot move left unless entering $S(n-1)$. So staying in $S(n,r)$, the lamplighter can move freely in the interval $[-z,z]$, switching lamps on appropriately. If we require the norm to equal $n$, then the position is $z$ or $-z$. Note that in this case, when $z_k=0$, the element $g_k$ of norm $n+z$ is a dead-end of $\Z\wr L$ of depth $\min(b,-a)$.
\end{proof}

The following lemma describes which elements of the sphere $S(n)$ belong to the infinite component of $B(n-1)^c$.

\begin{lemma}\label{infcomp}
Let $g$ belong to $S(n)$, with $z \geq 0$. Then:
\begin{enumerate}
\item If $a=0$, then $g$ belongs to $S(n)^\infty$.
\item If $a<0$ and $z=b$, then $g$ belongs to $S(n)^\infty$.
\item If $a<0\leq z <b$, then $g$ belongs to $S(n)^\infty$ iff $z \geq |a|$.
\end{enumerate}
\end{lemma}

\begin{proof}
If $a=0$, the path $g_k=gu^{-k}$, interpreted as ``go straight to $-\infty$'', has length $n+k$, so $g$ belongs to the infinite component $S(n)^\infty$ of $B(n-1)^c$. If $a<0$ and $z=b$, then the path $g_k=gu^k$, ``go straight to $+\infty$'', has length $n+k$ so $g$ belongs to $S(n)^\infty$. If $a<0\leq z <b$, then $g_k=gu^{-k}$ has length:
$$|g_k|=\left\{\begin{array}{ll} n+k &\textrm{ if } 0 \leq k \leq z, \\ n+2z-k &\textrm{ if } z \leq k \leq z+|a|, \\ n+k-2|a| &\textrm{ if } k \geq z+|a|. \end{array} \right. $$
If $z \geq |a|$, we have $|g_k|\geq n$ for all $k \geq0$, so $g$ belongs to $S(n)^\infty$. If not, the proof of Lemma \ref{connectedcomp} shows that $g \notin S(n)^\infty$.
\end{proof}

These two Lemmas enable us to prove the following. 

\begin{proposition}\label{llcomp}
Let $L$ be a finite group. For the lamplighter group $\Z \wr L$ on a line with respect to the ``switch-walk-switch'' generating set $L\{\pm1\} L$, the set $S(n,r)^\infty$ has at least $2|L|^{n-r}$ connected components while $1 \leq r+1 \leq n$.
\end{proposition}

This is the second statement in
Theorem \ref{nonconnected-intro}.(1).

\begin{proof} 
An element $g$ such that $a=0$ and $b=z=n$ belongs to $S(n)^\infty$ by Lemma \ref{infcomp}. Moreover, by Lemma \ref{connectedcomp}, all elements $g'$ in $S(n)^\infty$ in the connected component of $g$ in $S(n,r)^\infty$ have a lamp function $f'$ such that $f'(t)=f(t)$ for all $t$ in $[0,n-r-1]$. Since we have $|L|^{n-r}$ possibilities for the choice of the values of $f$ on this interval, there are at least this many connected components. In fact, at least twice more by symmetry.
\end{proof}

\subsection{Connection thickness}

\begin{theorem}\label{llconn}
Let $L$ be a finite group. Assume $n \geq 2$. In $\Z \wr L$ endowed with the ``switch-walk-switch'' generating set $L\{\pm 1\}L$, the annulus $S(n,n+2)^\infty$ is connected, whereas the annulus $S(n,n+1)^\infty$ has 3 connected components. Therefore the connection thickness is $$\mathrm{th}_{\Z \wr L, L\{\pm 1\} L}(n)=n+2.$$
\end{theorem}

This is the first statement of
Theorem \ref{nonconnected-intro}.(1).

\begin{proof}
First consider $S(n,n+1)^\infty$, and choose $g$ such that $a=0$ and $b=z=n$. The lamplighter cannot move left of position $n$ without letting a lamp turned on in $b_1 \geq b$ (otherwise, the path would enter $B(n-1)$). Then, while its position $z$ is between $0$ and $b_1$, the norm is $2b_1-z \geq 2n-z$. The lamplighter can reach position $-1$ (if $b_1=b$), but then the norm equals $2n+1$, and the lamplighter cannot come back to $0$ without switching off the lamp at $-1$ (otherwise, the path would enter $B(2n+2)$). This shows that $g'$ of norm $n$ belongs to the same connected component of $S(n,n+1)^\infty$ as $g$, if and only if $a'=0$ and $b'=z'=n$. By symmetry, there is another connected component containing elements with $a=z=-n$ and $b=0$. The annulus $S(n,n+1)^\infty$ has at least a third connected component since these two do not exhaust all $S(n)^\infty$. That there is no more than 3 connected components will be a consequence of the proof of the next statement.

Now we prove that $S(n,n+2)^\infty$ is connected. The first and main step is to prove that if $g$ belongs to $S(n)^\infty$ with $z \geq 0$, then there is a path in $S(n,n+2)^\infty$ from $g$ to $g'$ with $a'=0$ and $b'=z'=n$. If $b=n$, this is obvious, so we assume $b<n$. Moreover by Lemma \ref{infcomp}, one of the following two possibilities occurs: either 
$z \geq |a|$, or 
$z=b$.

This gives the two following cases:
\begin{itemize}
\item[(i)] $z \geq |a|$ and $b < n$.
\item[(ii)] $0 \leq z < |a|$, $z = b$ and $b <n$.
\end{itemize}
Start with case (i). Then $0 \leq |a| \leq z \leq b \leq n-1$. We prove that there is a path from $g$ to $g'$ satisfying the same relations, with $b'>b$, which implies the main step by induction. See Figure \ref{casi} for a picture which will hopefully make the argument easier to follow.

\begin{figure}
\begin{centering}
\begin{tikzpicture}
\draw[very thick,->] (-14/3,0) -- (30/3,0) node[below]{$\mathbb{Z}$};
\draw[very thick,->] (-14/3,0) -- (-14/3,28/3) node[above]{$|g_k|$};

\draw[very thick](0,1/10)--(0,-1/10) node[below]{$0$};
\draw[very thick](14/3,1/10)--(14/3,-1/10) node[below]{$z$};
\draw[very thick](26/3,1/10)--(26/3,-1/10) node[below]{$b$};
\draw[very thick](28/3,1/10)--(28/3,-1/10) node[below,cyan]{$b_1$};
\draw[very thick](-6/3,1/10)--(-6/3,-1/10) node[below]{$a$};
\draw[very thick](-12/3,1/10)--(-12/3,-1/10) node[below,blue]{$a_1$};
\draw[very thick](-10/3,1/10)--(-10/3,-1/10) node[below,red]{$a_2$};

\draw[very thick,dashed] (0,0)--(-2,1/4);
\draw[very thick,dashed] (-2,1/4)--(26/3,5/3);
\draw[very thick,dashed] (26/3,5/3)--(14/3,11/3) node[below]{$g$};
\draw[very thin,dashed] (28/3,11/3)--(-14/3,11/3) node[left]{$n$};
\draw[very thick] (14/3,11/3)--(0,18/3) node[below]{A};
\draw[very thick] (0,18/3)--(-6/3,15/3) node[below]{B};
\draw[very thick] (-6/3,15/3)--(-12/3,18/3) node[left]{C};
\draw[very thick,blue] (-12/3,18/3)--(0,24/3) node[below]{D};
\draw[very thick,blue] (0,24/3)--(26/3,11/3) node[below]{E};
\draw[very thick,blue] (26/3,11/3)--(28/3,12/3) node[right]{F};
\draw[very thick,cyan] (28/3,12/3)--(0,26/3) node[above]{G};
\draw[very thick,cyan] (0,26/3)--(-12/3,20/3) node[left]{H};
\draw[very thick,cyan] (-12/3,20/3)--(-10/3,19/3) node[below]{I};
\draw[very thick,red] (-10/3,19/3+1/10)--(0,24/3+1/10) node[above]{J};
\draw[very thick,red] (0,24/3+1/10)--(26/3,11/3+1/10) node[above]{K};
\end{tikzpicture}
\end{centering}
\caption{Illustration of case (i) (with $b-z$ even) in the proof of Proposition \ref{llconn}.}\label{casi}
\end{figure}
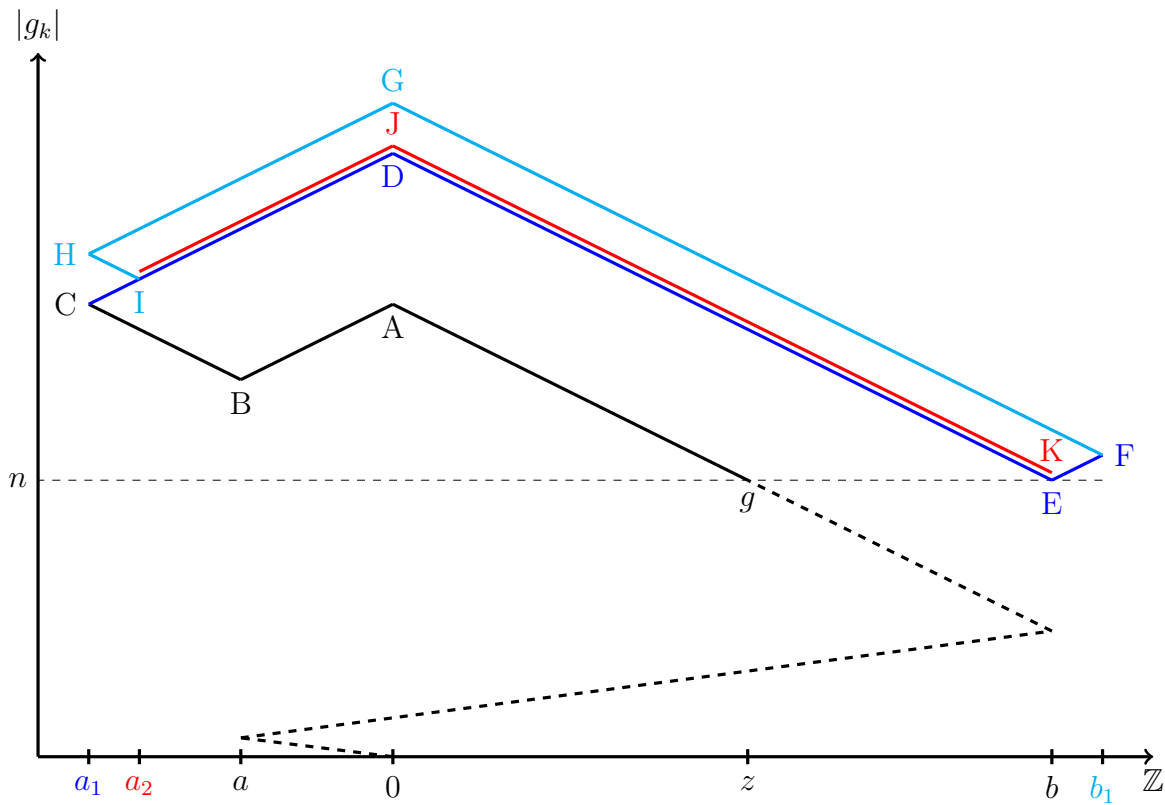

From $g$, the lamplighter keeps $b$ fixed and moves left to position $0$, where 
\[\tag{A}
|g_k|=n+z \leq 2n. 
\]
Then it continues left to position $a$, where 
\[\tag{B}
|g_k|=n+z+a\geq n, 
\]
and continues left again until position:
\[\tag{C}
a_1=\left\{ \begin{array}{lll} 
a-\frac{b-z}{2}   & \text{ if $b-z$ even,} & \text{ where }  |g_k|=n+\frac{b+z}{2}+a \in [n,2n],\\ 
a-\frac{b-z+1}{2} & \text{ if $b-z$ odd,}  & \text{  where } |g_k|=n+\frac{b+z+1}{2}+a \in [n,2n].
\end{array} \right. 
\]
\newcommand{\smdd}[2]{\Bigg\lbrace \begin{array}{ll} #1 & \text{if $b- z$ even} \\ #2 & \text{if $b - z$ odd} \end{array} }
Then the lamplighter keeps the lamp at $a_1$ on, and moves back right. At position $0$, we have 
\[\tag{D}
|g_k|=\smdd{n+b}{n+b+1}
\]
and at position $b$, 
\[\tag{E}
 |g_k|=\smdd{n}{n+1}
\]
Step right once more to $b_1=b+1$, where 
\[\tag{F}
|g_k|=\smdd{n+1}{n+2}
\]
Then keep the lamp at $b_1$ turned on and move left. At position $0$, we have 
\[\tag{G}
|g_k|=\smdd{n+b+2}{n+b+3}  
\]
At position $a_1$, we have:
\[\tag{H}
|g_k|=
\left\{ 
\begin{array}{llll} 
n+b+2+a_1 &= n+2+a+\frac{b+z}{2}     
& \text{ if $b-z$ even,}\\ 
n+b+3+a_1 &= n+3+a+\frac{b+z+1}{2} 
& \text{ if $b-z$ odd.}
\end{array} \right. 
\]
Observe that $a_1 <0$. Indeed, $a_1 =0$ implies, in the even case, $a=0$ and $b = z$ (which was already treated) or, in the odd case, $a=0$ and $b = z -1$ (which is not possible). Move right switching off the lamp at $a_1$ to set $a_2=a_1+1$ (so $a_2 \leq 0$), where 
\[\tag{I}
|g_k|=\smdd{n+b+1+a_1}{n+b+2+a_1} 
\]
Keep $a_2$ fixed and move right to position $0$, where 
\[\tag{J}
|g_k|=\smdd{n+b}{n+b+1} 
\]
then right to final position $b_1= b+1$, where 
\[\tag{K}
|g'|=n. 
\]
The end $g'$ of the path satisfies $a'=a_2=a_1+1$, $b'=b_1=b+1>b$ and $z'=b$ (even case) or $b+1$ (odd case). One can check that $g'$ satisfies case (i). 

One can check that the two steps where the norm is close to the admitted bounds is B (minimal norm) and G (maximal norm). In B, the norm is $n+|z|-|a| \geq n$. In G, the norm is (in the even case) $\leq n+b+2 < 2n+2$ or (in the odd case) $\leq n+b+3 \leq 2n+2$. 

There remains to treat case (ii): $0 \leq z = b < |a|$. See Figure \ref{casii} for a simplified picture of the path.

The lamplighter moves right to $b_1=|a|$ (A), where $|g_k|=n-a-z\in [n,2n]$. Keeping the lamp at $b_1=|a|$ on, he moves left to position $0$ (B), where $|g_k|=n-2a-z\in [n+|a|,2n]$, and then to position $a$ (C), where $|g_k|=n-a-z \in [n,2n]$. Switching off lamps, the lamplighter moves right to $a_1=-z$ (D), where $|g_k|=n$. Keeping the lamp at $a_1=-z$ on, he moves right to position $0$ (E), where $|g_k|=n+z$, and finally to position $z$ (F), where $|g'|=n$, with $a'=a_1=-\g$, $b'=b_1=|a|$ and $z'=z$. We have $z' = |a'|$, so $g'$ qualifies for case (i).

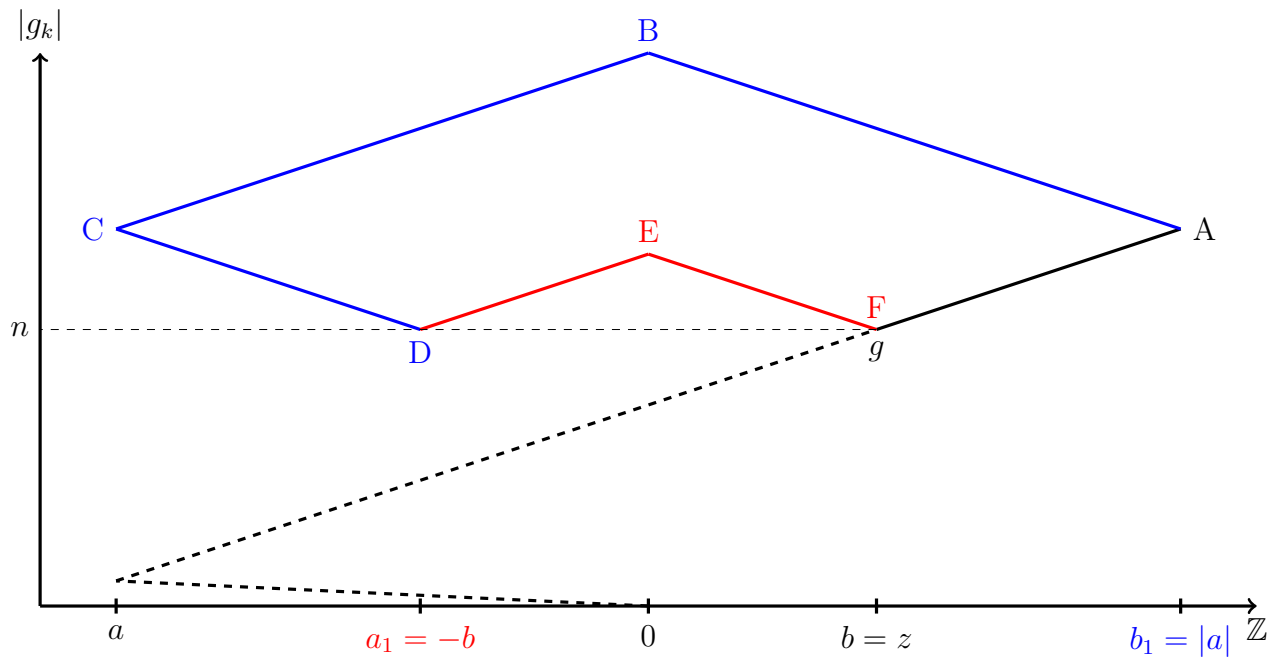
\begin{figure}
\begin{centering}
\begin{tikzpicture}
\draw[very thick,->] (-8,0) -- (8,0) node[below]{$\mathbb{Z}$};
\draw[very thick,->] (-8,0) -- (-8,22/3) node[above]{$|g_k|$};

\draw[very thick](0,1/10)--(0,-1/10) node[below]{$0$};
\draw[very thick](3,1/10)--(3,-1/10) node[below]{$b=z$};
\draw[very thick](7,1/10)--(7,-1/10) node[below,blue]{$b_1=|a|$};
\draw[very thick](-7,1/10)--(-7,-1/10) node[below]{$a$};
\draw[very thick](-3,1/10)--(-3,-1/10) node[below,red]{$a_1=-b$};

\draw[very thick,dashed] (0,0)--(-7,1/3);
\draw[very thick,dashed] (-7,1/3)--(3,11/3) node[below]{$g$};
\draw[very thick] (3,11/3)--(7,15/3) node[right]{A};
\draw[very thick,blue] (7,15/3)--(0,22/3) node[above]{B};
\draw[very thick,blue] (0,22/3)--(-7,15/3) node[left]{C};
\draw[very thick,blue] (-7,15/3)--(-3,11/3) node[below]{D};
\draw[very thick,red] (-3,11/3)--(0,14/3) node[above]{E};
\draw[very thick,red] (0,14/3)--(3,11/3) node[above]{F};
\draw[very thin,dashed] (3,11/3)--(-8,11/3) node[left]{$n$};

\end{tikzpicture}
\end{centering}
\caption{Illustration of case (ii) in the proof of Proposition \ref{llconn}.}\label{casii}
\end{figure}

We have proved the first main step, that all elements of $S(n)^\infty$ with $z \geq 0$ belong to the same connected component of $S(n,n+2)^\infty$. By symmetry, the same is true for elements with $z \leq0$. There remains to show a path from one to the other. If $n$ is a multiple of $3$, take $g$ with $a=\frac{-n}{3}$ and $b=z=\frac{n}{3}$. When the lamplighter moves left to position $a'=a=z'=-z$, keeping $b'=b$, the norm is bounded $|g_k|\in [n,\frac{4n}{3}]$. This is the required path (which can easily be adapted when $n$ is not a multiple of $3$).

To finish the proof of the proposition, 
we observe that the paths we constructed connecting elements were always in $S(n,n+1)^\infty$, except in the odd case of case (i) where it could happen that $n+b+3 = 2n+2$. However, this happens exactly when $b=n-1$ and implies $a=0$ and $z = n-2$. Iterating ``case (i)'' leads necessarily to this situation since $b$ increases by $1$ at every step (and $b=n-1 \implies a=0$ and $z =n-2$). It is straightforward to check that the lamp configuration on $[0,n-2]$ can be arbitrarily modified staying in $S(n,n+1)^\infty$. This shows that there are at most 4 (by symmetry) connected components of $S(n,n+1)^\infty$. 

To show that there are 3 components, it is only left to check that the element $g$ with $a =-n+1, z = -n+2, b=0$ is connected to $g'$ with $a'=0,z'= n-2$ and $b'=n-1$.  To do so, consider again an element $g''$ with $-a'' = b'' = z'' = \tfrac{n}{3}$. Applying the ``case (i)'' strategy to $g''$, one will show it is connected to $g'$. On the other hand, $g''$ is easily connected (as above) to its symmetric element (with a negative lamplighter position). Applying the strategy symmetric to ``case (i)'' will give a path to $g$. This shows that there is a path between $g$ and $g'$. Together with the first paragraph of this proof, this shows that $S(n,n+1)^\infty$ has exactly $3$ connected components.
\end{proof}
In case the reader wonders what happens for $n=1$, we mention that $S(1,2)$ is connected.

\subsection{Most of the sphere is in $S(n)^\infty$ .}

\begin{lemma}\label{blambda}
For $\l \in [0,1]$, there exists $c>0$ such that:
$$ |\{g \in S(n)|0 \leq z \leq b \leq \l n \}| \leq cn^2|L|^{\frac{1+\l}{2}n}.$$
\end{lemma}

\begin{proof}
We have $2|a|+b \leq n$ since $z \geq 0$ and $b \leq \l n$, so $|a|+b \leq \frac{1+\l}{2}n$. An element $g$ as required is described by $a \in [-n,0]$, $z \in [a,b]$ and at most $\frac{1+\l}{2}n$ lamps in $L$ in the interval $[a,b]$.
\end{proof}

\begin{lemma}\label{sizeofsphere}
In $\Z \wr L$ with $L$ finite of size $\ell=|L|$ endowed with the switch-walk-switch generating set, one has:
$$  |S(n)| =2\left(\ell+1\right)^2\ell^{n-1}+o\left(\ell^n\right).$$
\end{lemma}

Recall the Landau notation that $o(\ell^n)$ means a sequence tending to zero times $\ell^n$.

\begin{proof}
We first estimate the number of elements in $S(n)$ satisfying $z >0$. We count first the elements that satisfy $a=0$ and discuss on $b$, which enforces the value of $z$. When $b=z=n$, there are $\ell^{n+1}$ possibilities for the lamp configuration. When $b=n-i$ with $0<i$ small, which forces $z=n-2i$, there are $(\ell-1)\ell^{n-i}$ possibilities for the lamp configuration. Note that the lamp in $b$ cannot be trivial. So the number of elements in $S(n)$ with $z>0$ and $a=0$ is
\[
\ell^{n+1}+\sum_{i=1}^{\e n} (\ell-1)\ell^{n-i}+E_1 =(\ell+1)\ell^n+E_2
\]
where $E_i$ are error terms satisfying $|E_i|\leq c n^2 \ell^{\left(1-\frac{\e}{2}\right)n}$ by Lemma \ref{blambda}.

Now for a fixed $a<0$, again discuss on $b$. If $b=z=n-2|a|$, there are $(\ell-1)\ell^{n-|a|}$ possibilities for the lamp configuration, the lamp in $a$ being non-trivial. For $b=n-2|a|-i$ with $i>0$, which forces $z=n-2|a|-2i$, there are $(\ell-1)^2\ell^{n-|a|-i-1}$ possibilities for the lamp configuration. So the number of elements in $S(n)$  with $z>0$ and a given $a<0$ is
$$(\ell-1)\ell^{n-|a|}+\sum_{i=1}^{\e n}(\ell-1)^2\ell^{n-|a|-i-1}+E_3=(\ell-1)(\ell+1)\ell^{n-|a|-1}+E_4.$$

In total, we sum for $0 \leq |a| \leq \e n$ and get
\[
|S(n)\cap\{z>0\}|=(\ell+1)\ell^n+\sum_{|a|=1}^{\e n}(\ell-1)(\ell+1)\ell^{n-|a|-1}\pm cn^3 \ell^{\left(1-\frac{\e}{2}\right)n}=\ell^{n-1}(\ell+1)^2+o(\ell^n).
\]

Finally, the result has to be doubled to take into account elements with $z<0$, whereas those with $z=0$ are negligible.
\end{proof}

We are now able to prove that almost all the sphere $S(n)$ is in the infinite component of $B(n-1)^c$. Namely:

\begin{proposition}\label{relsize}
There exists $\a<1$ such that:
$$ \frac{|S(n)\setminus S(n)^\infty|}{|S(n)|} \leq \a^n.$$
\end{proposition}

\begin{proof}
Let $g$ belong to $S(n)\setminus S(n)^\infty$ with $z \geq0$. We claim that $b \leq \frac{5}{8}n$. Indeed, if $b \geq \frac{n}{2}$, then $|a| \leq \frac{n}{4}$, so by Lemma \ref{infcomp} (3), $z \leq |a| \leq \frac{n}{4}$, which implies $2b-\frac{n}{4} \leq 2|a|+2b-z =n$ and $b \leq \frac{5}{8}n$. By Lemma \ref{blambda}, $|S(n)\setminus S(n)^\infty|\leq c n^2 |L|^{\frac{13}{16}n}$, and the conclusion follows by Lemma \ref{sizeofsphere}.
\end{proof}

\begin{remark}\label{straightconnection}
Let us say that $g \in S(n)$ is straightly connected to infinity if there exists a geodesic $g_k$ such that $g_0=g$ and $|g_k|=n+k$ for all $k$. Denote $S(n)^{s\infty}$ the set of such $g$'s. The proof of Lemma \ref{infcomp} essentially shows that $g \in S(n)^{s\infty}$ exactly in cases (1) or (2). Similarly to Lemma \ref{sizeofsphere}, we can compute precisely the number of elements of length $n$ straightly connected to infinity.

Let us count the number of elements in $S(n)^{s\infty}$ with $z>0$. Under the conditions $a=0$ and $b=n-i$, which enforce the value $z=n-2i$, there are $\ell^{n+1}$ possible lamp configurations when $i=0$ and only $\ell^{n-i}(\ell-1)$ when $i>0$ for the lamp at $b$ cannot be trivial. Under the condition $a<0$, straight connection to infinity forces $z=b=n-2|a|$ and there are $(\ell-1)\ell^{n-2|a|}$ possible lamp configuration. All in all
\begin{eqnarray*}
|S(n)^{s\infty}|&=&2\left(\ell^{n+1}+\sum_{i=1}^\infty (\ell-1)\ell^{n-i}+\sum_{|a|=1}^{\infty} (\ell-1)\ell^{n-2|a|}  \right)+o(\ell^n)\\ &=&2\left(\ell+1+\frac{1}{\ell+1}\right)\ell^n+o(\ell^n).
\end{eqnarray*}

Comparing with Lemma \ref{sizeofsphere}, it follows that the sequence $\frac{|S(n)^{s\infty}|}{|S(n)|}$ converges to a constant in $(0,1)$ and this constant is $\frac{20}{27}$ for $L=\Z_2$ and tends to $1$ when $\ell\rightarrow \infty$. This means that a positive fraction of the sphere is not straightly connected to infinity (and a larger positive fraction is).

We wonder if there are examples of finitely generated groups (or of vertex transitive graphs) where this sequence or a subsequence of it could take arbitrary small values ? See Question \ref{qustra}.
\end{remark}

\subsection{Normalised entropy of partitions}\label{entropy}

Let $\tilde{\Pi}(n,r)$ denote the partition of $S(n,r)$ into connected components, and $\Pi(n,r)$  (resp. $\Pi(n,r)^\infty$) the inherited partition of $S(n)$ (resp. $S(n)^\infty$). Denote $\pi(g)\in \Pi(n,r)$ the component of $g$ in the partition. We have $\pi(g)=\pi(g')$ if and only if there is a path from $g$ to $g'$ in $S(n,r)$.

The Shannon entropy of a partition $\Pi$ of a space endowed with a probability $\mu$ is 
$$H(\Pi)=-\sum_{\pi \in \Pi} \mu(\pi)\log \mu(\pi). $$
We consider $H(\Pi(n,r))$ where $S(n)$ is equipped with the normalised counting measure.

It is well-known that $0\leq H(\Pi) \leq \log |\Pi|$ where $|\Pi|$ is the number of components in the partition $\Pi$. Moreover $H(\Pi)=0$ if and only if one component weights the full mass, whereas $H(\Pi)=\log|\Pi|$ if and only if all set in the partition have the same cardinality, see \cite{Sha}. This motivates the definition of normalised entropy by $h(\Pi)=H(\Pi)/\log|\Pi| $.

The normalised entropy of spheres of a fixed thickness $r$ in $\Z \wr L$ is essentially maximal. More precisely:

\begin{proposition}\label{entmax}
The normalised entropy of spheres of $\Z\wr L$ with switch-walk-switch generating set satisfies
$$h(\Pi(n,r)) \sim h(\Pi(n,r)^\infty) \underset{n\rightarrow \infty}\longrightarrow 1, $$
while $r \leq \e n$ for some $\e < \frac{1}{8}$.
\end{proposition}

Note that in a free group with free generating set, one has precisely $h(\Pi(n,r))=1$ for all $n,r\geq 0$.

\begin{lemma}\label{partitionsize}
The numbers of components of the partitions satisfy:
$$|\Pi(n,r)| \asymp |\Pi(n,r)^\infty| \asymp |L|^{n-r}, $$
while $r \leq \e n$ for some $\e < \frac{1}{8}$.
\end{lemma}

\begin{proof}[Proof of Lemma \ref{partitionsize}]
Set $S_1=\{g \in S(n)|z-r \geq 0\}$ and $S_2=S(n)\setminus S_1$. By Lemma \ref{connectedcomp}, we have $|\pi(g)|=|L|^{r+1}$ for $g$ in $S_1$. On the other hand, for $g$ in $S_2$, we have $z \leq \e n$ so $b\leq \frac{1+\e}{2}n$. This implies $|S_2| \leq cn^2|L|^{\frac{3+\e}{4}n}$ by Lemma \ref{blambda} and that $|S_1|$ is of order $|L|^n$ by Lemma \ref{sizeofsphere}. We conclude by Lemma \ref{sizeofsphere} and the inequalities $\frac{|S_1|}{|L|^{r+1}} \leq |\Pi(n,r)| \leq \frac{|S_1|}{|L|^{r+1}}+|S_2| $, where the left-hand-side term is leading by the condition on $\e$ . The same is true for $\Pi(n,r)^\infty$ by Proposition~\ref{relsize}.
\end{proof}

\begin{proof}[Proof of Proposition \ref{entmax}]
It is well-known that the partition $\Pi'$ obtained from $\Pi$ by taking $\pi \in \Pi$ and replacing it by individual classes $\{g\}_{g \in \pi}$ has bigger entropy. Likewise merging classes of a partition reduces entropy. Hence, with the same notations, we have:
\[
\begin{array}{c}
\displaystyle -\frac{1}{|S(n)|}\left( \sum_{g \in S_1} \log \frac{|L|^{r+1}}{|S(n)|} + |S_2| \log \frac{|S_2|}{|S(n)|} \right) \\
\rule{.1pt}{1.5ex}\,\wedge \\
\displaystyle  H(\Pi(n,r)) \\
\rule{.1pt}{1.5ex}\,\wedge\\
\displaystyle -\frac{1}{|S(n)|}\left( \sum_{g \in S_1} \log \frac{|L|^{r+1}}{|S(n)|} +\sum_{g \in S_2} \log \frac{1}{|S(n)|} \right). 
\end{array}
\]
Again, the sum indexed by $S_1$ is the leading term, so $H(\Pi(n,r))\sim \log \frac{|S(n)|}{|L|^{r+1}}$. We conclude by Lemmas \ref{sizeofsphere} and \ref{partitionsize}. The same is true for $\Pi(n,r)^\infty$ by Proposition~\ref{relsize}.
\end{proof}

\section{The lamplighter on a tree with finite lamps}\label{slltree}

In this section, we consider lamplighter groups with finite lamp group $L$, and where the base group has a tree for Cayley graph. Our aim is to prove Theorem \ref{nonconnected-intro}.(2). It will result of the concatenation of Corollary \ref{cornoncon} which gives an estimate on the number of connected components in $S(n,r)^\infty$ and Proposition \ref{treeconnection} which bounds from above the connection thickness.

We write $T=T_d$ for a group with generating set $S=S_d$ such that the Cayley graph $Cay(T_d,S_d)$ is a $d$-regular tree, e.g. the free product of $d$ involutions, or the free group of rank $d/2$ for even $d$. We consider the lamplighter group $T \wr L$, with the ``switch-walk-switch'' generating set $LSL$. The case of a line $d=2$ was studied in the previous section, and we assume here that $d \geq 3$.

For an element $g=(\g,f)$, denote $C(g)$ the minimal subtree of $T$ containing $e$, $\g$ and the support of the lamp function $f:T \rightarrow L$.

\begin{lemma}\label{treewl}
The word length of an element $g=(\g,f)$ of $T \wr L$ is $|g|=2|C(g)|-d(e,\g)$, where $|C(g)|$ is the number of edges in $C(g)$.
\end{lemma}

\begin{proof}
By Proposition \ref{propTSP}, the norm of $g$ is the solution of the travelling salesman problem from $e$ to $\g$ and visiting all vertices of $\mathrm{supp}(f)$. In a tree, the solution is such that all edges of $C(g)$ have to be crossed twice, except those on the geodesic from $e$ to $\g$.
\end{proof}

For a subset $A$ of a graph, denote $\partial A$ the set of boundary vertices of $A$ (those with a neighbour outside $A$), and $Int(A)=A \setminus \partial A$ the set of interior vertices. Also denote $B(c,r)$ the ball of centre $c$ and radius $r$.

\begin{lemma}\label{treedeadend}
Let $g$ belong to $S(n)$, then $g$ is in a dead-end component (i.e. $g \notin S(n)^\infty$) if and only if $B_T(e,d(e,\g)) \subset Int(C(g))$. In particular, if $\g$ belongs to $\partial C(g)$, then $g$ belongs to $S(n)^\infty$.
\end{lemma}

\begin{proof}
If $\g$ belongs to $\partial C(g)$ and the lamplighter moves outside of $C(g)$, then $|g|$ increases. So $g$ belongs to $S(n)^\infty$. If $\g$ belongs to $Int(C(g))$, then $|g|$ increases by one when the lamplighter steps in the direction of $e$, and decreases by one when he steps away. If the ball  $B_T(e,d(e,\g))$ is in the interior of $C(g)$, then the lamplighter cannot leave it without $g$ entering $S(n-1)$. Otherwise, he can reach a boundary position.
\end{proof}

This argument shows that an element $g$ is a dead-end of depth $\geq r$ if and only if $\g=e$ and $B_T(e,r) \subset Int(C(g))$.

Let $v$ be a vertex of the tree $T$. The subtree at $v$, denoted $T_v$, has for vertex set all $w$ such that $v$ belongs to the geodesic between $e$ and $w$. 

\begin{proposition}\label{lowlog}
Let $T'$ denote a subtree at a neighbour of $e$. Let $r<R$. Take $g$ such that $\g=e$, $C(g)$ does not intersect $T'$ and $B_T(e,R)\setminus T' \subset C(g)$. Then all elements $g'$ in the connected component of $g$ in $S(n,r)^\infty$ satisfy $f'(v)=f(v)$ for all $v$ in $C(g)$ with $d(e,v) \geq r+1$.
\end{proposition}

\begin{proof}
The element $g$ belongs to $S(n)^\infty$ by Lemma \ref{treedeadend}. By Lemma \ref{treewl}, the lamplighter must first step into $T'$, otherwise entering $S(n-1)$. He can visit $s$ sites in $T'$ and come back to $e$. There $|g_k|=n+2s \leq n+r$, staying in $S(n,r)$. He can then move into $C(g)$, with $|g_k|=n+2s-d(e,\g_k)$. Since $2s\leq r <R$, the lamplighter cannot reach a vertex $v$ not in $T'$ with $d(e,v)\geq r+1$.
\end{proof}

\begin{corollary}\label{cornoncon}
There exists positive constants $c_1,c_2$ depending only on $d$ such that $S(n,r)$ has at least $|L|^{c_1n}$ connected components for $r \leq \log_{d-1}(n)-c_2$.
\end{corollary}

\begin{proof}
Assume there exists an integer $r\geq 0$ such that
\[
n = |B_T(e,r) \setminus T'| = \frac{ (d-1)^r-1}{d-2}.
\]
Let $g$ be such that $C(g) = B_T(e,r) \setminus T'$ and $\g = e$. 
By Proposition \ref{lowlog}, there are at least $|L|^{(d-1)^{r+1}}$ connected components in $S(n,r)$ containing elements of this form. A generic $n$ is treated similarly.
\end{proof}

\begin{proposition}\label{treeconnection}
There exists a constant $c$ such that the annuli $S(n,\log_{d-1}(n)+c)^\infty$ are connected, so $\textrm{th}_{T\wr L, LSL}(n) \leq \log_{d-1}(n)+c$.
\end{proposition}

To ease notations, we set $h(\g)=d_T(e,\g)$ and call height the distance to the origin in the tree. A vertex $\g\neq e$ has exactly one neighbour of smaller height, which we call its parent, and $d-1$ of bigger height, its children. 
Two vertices are said to be siblings if they have the same parent.
The origin $e$ has $d$ children and, by convention, is its own parent. The $k$ ancestor is the parent of the $(k-1)$ ancestor (and the $0$ ancestor is the current vertex). In our representations of trees, the parents are above the children.

\begin{proof}[Proof of Proposition \ref{treeconnection}]
We will prove this proposition when $n=2|B_T(e,R)|-R$ for some $R$, where absolute value denotes the number of edges. We prove that $S(n,R)^\infty$ is connected, and clearly $R \leq \log_{d-1}(n)+c$. 

(When $n$ is not of the chosen form, we should replace $B_T(e,R)$ by an appropriate subset containing $B_T(e,R-1)$ in the definition of $\bar g$. The proof goes the same way, with more tedious details.)

The reader should keep in mind that for such an $n$, if $g$ belongs to $S(n)$, then $h(\g)$, $R$ and $n$ have the same parity.

For $\ell_0 \in L$ not trivial, set $\bar g = (e,\bar f)$ with $\bar f(v)=\ell_0$ for $v$ in $B_T(e,R)$ and $\bar f (v)=e$ otherwise. Then $\bar g$ belongs to $S(n+R)$. We take an arbitrary $g$ in $S(n)^\infty$, and construct a path from $g$ to $\bar g$ in $S(n,R+4)$.

An elementary situation is when $C(g)=B_T(e,R)$. The choice of $n$ forces $h(\g)=R$. Now as long as the lamplighter stays in the ball $B_T(e,R)$ and does not turn off lamps in its boundary $S_T(e,R)$, we have $|g_k|=n+R-h(\g_k)$ belongs to $[n,n+R]$, and so there is a path to $\bar g$.

We will often use the following lemma. Say there is a lamp off (resp. on) at vertex $v$ if $f(v)=e$ (resp. $\neq e$).
\begin{lemma}\label{loff}
Let $g=(\g,f)$ belong to $S(n)$, and let $\d$ denote the $(R+2)$ ancestor of $\g$.
\begin{enumerate}
\item If $h(\g)=R+2$, then $\d=e$ and there is a lamp off in $B_T(e,R)$.
\item If $h(\g)>R+2$, then there is a lamp off in $T_\d \cap B_T(\d,R+1)$.
\item If $h(\g)=R$ and there is a lamp on outside $B_T(e,R)$, then there is a lamp off in $B_T(e,R)$.
\item If $h(\g) \leq R-2$, then there is a lamp on outside $B_T(e,R-1)$.
\end{enumerate}
\end{lemma}
\begin{proof}
In the first case, if all the lamps are on, then $C(g) \supset B_T(e,R)$ and so $|g| > 2|B_T(e,R)|-R=n$, contradiction. The second case is similar, (mind radius $R+1$ because the vertex $\d$ has one less children than $e$) as well as the third and fourth.
\end{proof}

Given $g=g_0 \in S(n)^\infty$, our aim is to construct a sequence $(g_k)$ in $S(n)$ reaching an elementary situation such that $g_k$ and $g_{k+1}$ are connected by a path in $S(n,R+4)$ for each $k$. Our strategy is to construct $g_{k+1}$ from $g_k$ by an algorithm given below.

\begin{claimn} 
We may assume that $h(\g_0)\geq R$.
\end{claimn}
\begin{proof}[Proof of claim:]
Otherwise, $h(\g_0) \leq R-2$. As $g \in S(n)^\infty$, by Lemma \ref{treedeadend}, the ball $B_T(e,h(\g))$ is not included in $Int(C(g_0))$. Consequently, the set $\partial C(g_0) \cap B_T(e,R-2)$ is non-empty, hence contains a vertex $v$. By Lemma \ref{loff}.(4), there is $v'$ in $C(g_0)$ with $h(v')\geq R$. Then the lamplighter goes to $v$, lights up a lamp just outside $C(g_0)$, and then moves towards $v'$, stopping at height $h(\g_0)+2$. This configuration is in $S(n)$. On the way, $g$ remained in $S(n,R)$, because moving closer to $e$ by at most $R-2$ and adding only one extra edge to $C(g_0)$. Repeating this argument until reaching height $R$ justifies the claim. 
\end{proof}

From now on, we construct a sequence $(g_k)$ satisfying $h(\g_k) \geq R$ by the following algorithm. 

In order to measure the progress made by successive iterations of the algorithm, we define the following quantities for any vertex $\d$ of the tree. The maximal height of the lamp configuration (associated to $g= (\g,f)$) in $T_\d$ is denoted $h^{\max}_g(\d)$. The number of highest vertices is denoted $N^{\max}_g(\d)$. Finally, set 
\[
\e_g(\d) :=
\left\{\begin{array}{ll}
(h^{\max}_g(\d)-R-1)_+ & \text{when } \d \neq e,\\
(h^{\max}_g(\d)-R)_+ & \text{when } \d = e.
\end{array}\right.
\]
where $(x)_+$ is $x$ if $x \geq 0$ and $0$ otherwise.
This measures the excess of height of the configuration in the tree $T_\d$. We also denote by $\d_k$ the $(R+2)$ ancestor of $\g_k$.

Observe that if $h(\g_k)=h^{\max}_{g_k}(e)=R$, or equivalently if $\e(\d)=0$ for all $\d$ and $h(\g_k)= R$, then $g_k$ is an elementary configuration and we are done. We now describe the construction of $g_{k+1}$ from $g_k$.

The rough idea is to consider the intersection of $C(g_k)$ with the tree attached to the $(R+2)$ ancestor of the position $\g_k$. If $\g_k$ is maximal in this tree, we aim to reduce the height of this tree, if not, we aim to move the position higher. At times, the $(R+2)$ ancestor may change and the tree we consider may be enlarged, we will then reduce the height of a larger portion of $C(g_k)$. Eventually, getting a global minimal height enables us to reach an elementary configuration. Precisely, we distinguish two cases : 
\begin{itemize}
 \item[Case 1:] $R \leq h(\g_k) = h^{\max}_{g_k}(\d_k)$;
 \item[Case 2:] $R \leq h(\g_k) < h^{\max}_{g_k}(\d_k)$.
\end{itemize}

{\it Case 1:} if $h(\g_k)=h^{\max}_{g_k}(\d_k) \geq R$. If there is equality, this is an elementary configuration by the previous observation. So assume $h(\g_k)=h^{\max}_{g_k}(\d_k) \geq R+2$. By Lemma \ref{loff}.(1)-(2), there is a lamp off at $v$ in $T_{\d_k}\cap B_T(\d_k,R+1)$ (in fact in $B_T(e,R)$ if $h(\g_k)=R+2$). Keeping the lamp at $\g_k$ on, the lamplighter moves to light the lamp at $v$ and comes back to position $\g_k$ (he stays in $S(n,R+2)$ on the forward way and $S(n,R+4)$ on the way back, and at the end the norm is $n+2$). Then he steps up to his parent, turning off the lamp at $\g_k$. The norm is $n+1$. There are two possibilities. 

{\it Case 1a:} if some sibling of $\g_k$ is in $C(g_k)$, the lamplighter moves down to this sibling. This is $g_{k+1}$ in $S(n)$. Note that $\e(\d)$ has not changed for any $\d$ in the tree and $\d_{k+1}=\d_k$, but $N^{\max}_{g_{k+1}}(\d_k) < N^{\max}_{g_k}(\d_k)$.

{\it Case 1b:} if no sibling belongs to $C(\g_k)$, the lamplighter steps up to the grandparent of $\g_k$, turning off the lamp. This is $g_{k+1}$ in $S(n)$. Note that  $N^{\max}_{g_{k+1}}(\d_k) < N^{\max}_{g_k}(\d_k)$, except if $\g_k$ was unique among maximal vertices, in which case $\e_{g_{k+1}}(\d_k) < \e_{g_k}(\d_k)$. Mind that $h(\g_{k+1})=h(\g_k)-2$, and that $\d_{k+1}$ is different form $\d_k$ (except if both are $e$).

{\it In case 1a and 1b:} note that the excess at all other vertices has not increased: \linebreak $\forall \d, \e_{g_{k+1}}(\d) \leq \e_{g_k}(\d)$.

\begin{figure}

\begin{centering}

\begin{tikzpicture}
\draw (0,4)node[below]{$\dots$}--(0.5,5) node[right]{$e$};
\draw (0.5,5)--(1,4);
\draw (1,4)--(0.5,3);
\draw (1,4)--(1.5,3) node[right]{$\color{blue}{\delta_k} \color{black}{=}\color{red}{\delta_{k+1}}$};
\draw (1.5,3) node{$\bullet$};
\draw (1.5,3)--(1,2);
\draw[red, very thick] (1,2)--(0.5,1);
\draw (1.5,3)--(2,2);
\draw (2,2)--(1.5,1);
\draw (2,2)--(2.5,1);
\draw (1.5,1)--(2,0);
\draw[red] (2,0) node[below]{$\gamma_{k+1}$};
\draw (1,0) node{$\color{blue}\bullet$};
\draw (2,0) node{$\color{red}\bullet$};
\draw[blue, very thick] (1.5,1)--(1,0) node[below]{$\gamma_k$};
\draw (1.25,-2.5) node{Case 1a.};

\draw (4,5) node[right]{$e=\color{red} \delta_{k+1}$}--(3.5,4) node[below]{$\dots$};
\draw (4,5) node{$\color{red}\bullet$}--(4.5,4);
\draw (4.5,4)--(4,3);
\draw (4.5,4)--(5,3)node[right]{$\color{blue} \delta_k$};
\draw (5,3) node{$\color{blue}\bullet$};
\draw (5,3)--(4.5,2);
\draw[red, very thick] (4,1)--(4.5,2);
\draw (5,3)--(5.5,2) node[right]{$\color{red} \gamma_{k+1}$};
\draw (5.5,2) node{$\color{red}\bullet$};
\draw[blue, very thick] (5.5,2)--(5,1);
\draw[blue, very thick] (5.5,0) node[below]{$\gamma_k$}--(5,1);
\draw (5.5,0) node{$\color{blue}\bullet$};
\draw (5.5,2)--(6,1);
\draw (4.75,-2.5) node{Case 1b.};

\draw (7.5,5) node[right]{$e$} -- (7,4)node[below]{$\dots$};
\draw (7.5,5)--(8,4);
\draw (8,4)--(7.5,3);
\draw (8,4)--(8.5,3) node[right]{$\color{blue} \delta_k$};
\draw (8.5,3) node{$\color{blue}\bullet$};
\draw (8,2)--(8.5,3);
\draw (9,2)--(8.5,3);
\draw[red, very thick] (7.5,1)--(8,2);
\draw (9,2)--(8.5,1) node[right]{$\color{red} \delta_{k+1}$};
\draw (8.5,1) node{$\color{red}\bullet$};
\draw (9,2)--(9.5,1);
\draw (9.5,1)--(10,0)node[right]{$\color{blue} \gamma_k$};
\draw (10,0) node{$\color{blue}\bullet$};
\draw (8.5,1)--(8,0);
\draw (8.5,1)--(9,0);
\draw (8.5,-1)--(9,0);
\draw (9.5,-1)--(9,0);
\draw (9.5,-1)--(9,-2);
\draw (9.5,-1)--(10,-2) node[right]{$\color{red} \gamma_{k+1}$};
\draw (10,-2) node{$\color{red}\bullet$};
\draw (8.75,-2.5) node{Case 2a.};

\draw (11.5,5)--(11,4) node[below]{$\dots$};
\draw (11.5,5) node[right]{$e$}--(12,4);
\draw (12,4)--(11.5,3);
\draw (12,4)--(12.5,3)node{$\bullet$};
\draw (12.5,3) node[right]{$\color{blue}{\delta_k} \color{black}{=}\color{red}{\delta_{k+1}}$};
\draw (12.5,3)--(12,2);
\draw[red, very thick] (11.5,1)--(12,2);
\draw (12.5,3)--(13,2);
\draw (12.5,1)--(13,2);
\draw (13.5,1)--(13,2);
\draw (13.5,1)--(14,0)node[right]{$\color{blue} \gamma_k$};
\draw (14,0) node{$\color{blue}\bullet$};
\draw (12.5,1)--(12,0);
\draw (12.5,1)--(13,0)node{$\color{red}\bullet$};
\draw (13,0) node[right]{$\color{red} \gamma_{k+1}$};
\draw (13,0)--(13.5,-1);
\draw[blue, very thick] (13,0)--(12.5,-1);
\draw (12.5,-2.5) node{Case 2b.};
\end{tikzpicture}
\end{centering}

\caption{Illustration of the proof of Proposition \ref{treeconnection}. The black part is common to steps $k$ and $k+1$, the blue part concerns step $k$, the red part step $k+1$.}
\end{figure}
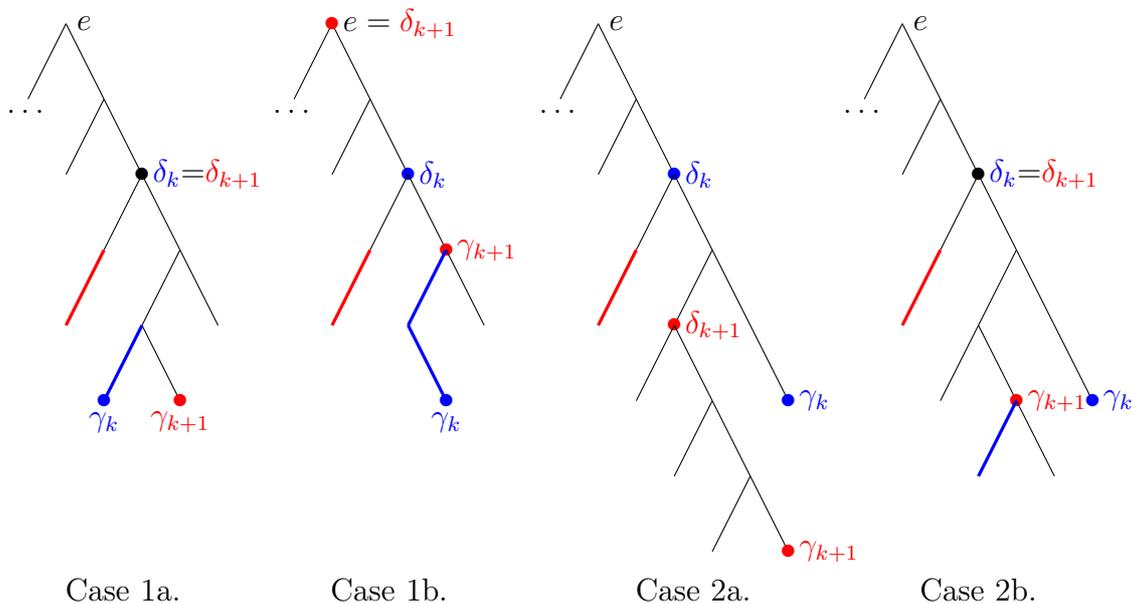

{\it Case 2:} if $R \leq h(\g_k) < h^{\max}_{g_k}(\d_k)$. Again, by Lemma \ref{loff}.(1)-(3) and excluding an elementary configuration, there is a lamp off above in $T_{\d_k}$. Keeping the lamp on at $\g_k$, the lamplighter moves to this lamp, turns it on, and goes down again, this time in the direction of a vertex of the configuration of maximal height in $T_{\d_k}$ (he can do this staying in $S(n,R+4)$). At height $h(\g_k)$, the norm is $n+2$. There are two possibilities.

{\it Case 2a:} if $h^{\max}_{g_k}(\d_k) \geq h(\g_k)+2$, the lamplighter stops on the way down at height $h(\g_k)+2$. This is $g_{k+1}$ in $S(n)$. Note that none of $h^{\max}_{g_k}(\d) = h^{max}_{g_{k_1}}(\d)$ and $\e_{g_k}(\d) = \e_{g_{k+1}}(\d)$. However, we have $h^{\max}_{g_{k+1}}(\d_{k+1})-h(\g_{k+1}) < h^{\max}_{g_k}(\d_{k})-h(\g_k)$ (observe that $\d_{k+1} \neq \d_k$ but both have the same maximal height). If the left-side is zero, the next step of the algorithm is case 1. It is case 2 otherwise.

{\it Case 2b:} if $h^{\max}_{g_k}(\d_k) = h(\g_k)+1$, the lamplighter goes to a maximum vertex $v$ in $T_{\d_k}$ (\ie $h(v) = h^{\max}_{g_k}(\d_k)$), where the norm is $n+1$. He steps up turning off the lamp behind him. This is $g_{k+1}$ in $S(n)$. Note that $N_{\max}(\d_k)$ decreased, except if the maximum vertex was unique, in which case $\e(\d_k)$ decreased (because $h(\g_k)+1 \geq R+1$ and in the equality case, $\d_k = e$).

In each of the four cases of the algorithm, at least one quantity decreases among all $N_{\max}(\d)$ and $\e(\d)$ (which is decreased each time some $N_{\max}(\d)$ "reaches zero"), except in case 2a. However, the difference of heights decreases at each step 2a, so the algorithm cannot get stuck there (it will end up in a ``case 1''). Eventually, the algorithm will reach a configuration with $\e(\d)=0$ for all $\d$, which is elementary.
\end{proof}

We use $\e(\d)$ in the above proof rather than $h_{\max}(\d)$ alone, because the configuration in the $R$ or $(R+1)$ levels of $T_{\d_k}$ is modified at each run of the algorithm. Therefore, the maximal height is not decreasing along time, contrary to the excess.

\section{The lamplighter on the ladder with finite lamps}\label{sec:ladder}

This section describes the case of the lamplighter on the ladder, namely the ``base'' group is $\Gamma = \Z \times \Z_2$, where $\Z_2=\Z/2\Z$. The aim is to prove Theorem \ref{teo-echelle-intro}, namely to show that there is a generating set with connected spheres and bounded retreat depth while there is another one with neither properties.

As before, $L$ denotes the ``lamp'' group. Elements of $G = \Gamma \wr L$ will be denoted by $(z,\e,f)$ where $z \in \Z$, $\e \in \Z_2$ and $f \in \oplus_\Gamma L$. The trivial elements in these groups are denoted by $0$ and the non-trivial element of $\Z_2$ is denoted by $1$. 

A generating set for $G = \Gamma \wr L$ which has neither connected spheres nor admits a bound on retreat depth is given by
\[
S_1 =\{ (0,0,l \delta_{(0,0)} + l' \delta_{(0,1)}  )\cdot (\pm 1, \e ,0) \cdot (0,0,l \delta_{(0,0)} + l' \delta_{(0,1)}  ) \text{ where } l,l' \in L  \text{ and } \e \in \Z_2 \}.
\]
In other words, one allows to do a ``switch-walk-switch'' but (and this is crucial) one may switch the two lamps that share the same $\Z$-coordinate simultaneously. The walk movement can be done simultaneously in $\Z_2$ and $\Z$.

Note that $\Z \wr (L^2)$ sits as an index $2$ subgroup of $G$. With the exception of the lone element $(0,1,0)$, the length of an element in $G$ can be computed by forgetting the $\Z_2$ coordinate on the ``base'' component. The value of the lamp is remembered because the ``lamp'' group is now $L^2$ and looking at the length of the corresponding element in $\Z \wr (L^2)$. From there, the arguments and statements of \S{}\ref{llline}, in particular Proposition \ref{llcomp} and Theorem \ref{llconn}, apply straightforwardly to the Cayley graph of $G$ with respect to $S_1$ and show it has neither connected spheres nor bounded retreat depth.

A generating set for $G = \Gamma \wr L$ which has connected spheres and admits a bound on retreat depth is the ``switch-walk-switch'' where the walk movement is for the summed generating set of $\Z \times \Z_2$ (\ie $\{(\pm 1,0), (0,1)\}$).

\begin{theorem}\label{techelle}
The group $G = (\Z \times \Z_2) \wr L$ for $L$ finite with the ``switch-walk-switch'' generating set $L\{(\pm 1,0), (0,1)\}L$ has connected spheres (the connection thickness is $\leq 24$) and bounded retreat depth (by $5$).
\end{theorem}

The remainder of this section is devoted to the proof of this theorem.

Take $g$ an element of $G=(\Z \times \Z_2) \wr L$. In the course of the travelling salesman path $c$ given in Proposition \ref{propTSP}, we may only remember the displacements which are made on the $\Z$ component. More precisely, let $\pi:\Z \times \Z_2 \twoheadrightarrow \Z$ be the projection. Call $\pi c$ the projected path. 

Let $s\mapsto \bar{c}(s) = \pi c (t_s)$ be the path obtained by removing all the pauses (\ie $\{t_s\}_{s \in \N}$ are the integers so that $\pi c(t_s+1) = \pi c(t_s) \pm 1$). We will describe completely the possible paths $\bar c$ obtained. A path $\bar{c}$ is said to backtrack whenever the direction of movement changes.

We call {\it lower line} the subset of $\G$ the elements of which have coordinates of the form $(z,0)$, and {\it upper line}, the subset of those with coordinates of the form $(z,1)$.

\begin{lemma}\label{3edge}
For any $g = (z,\e, f) \in (\Z \times \Z_2) \wr L$, there is a reduced representative word of $g$ so that the projected path $\bar{c}$ never uses an edge $3$ times or more. 

More precisely, the path $c$ solution of the travelling salesman problem on $\Z \times \Z_2$ associated to $g$ can be chosen so that the path $\bar c$ is a solution to the projected travelling salesman problem on $\Z$.
\end{lemma}

Following \S{}\ref{llline}, we set $a=\min\left(\pi (\textrm{supp}f) \cup \{0,z\}\right)$ and $b=\max\left(\pi (\textrm{supp}f) \cup \{0,z\}\right)$. When $z \geq 0$, which we can always assume by symmetry, the path $\bar c$ obtained in this lemma covers exactly once the edges in $[0,z]$ and exactly twice the edges in $[a,0]\cup [z,b]$.

\begin{proof}
Let $\bar{c}_0$ denote the projected path associated to a geodesic representative of $g$. 

By contradiction, we assume that there is at least one edge used at least three times. We consider $x_1,x_2 \in \Z$ such that $[x_1,x_2]$ is a maximal intervalle on which all edges are used at least three times. We show that we can replace the path $\bar c_0$ by a path of smaller or equal length using edges in $[x_1,x_2]$ at most twice. This suffices to show the first part of the lemma.

The path $\bar c_0$ is necessarily contained in $[a,b]$ and visits both $a$ and $b$.

Recall that our path $\bar c_0$ starts in $0$ and ends in some coordinate $z$. By symmetry, we may assume $0 \leq z$. There are 3 cases:
\begin{enumerate}
\item both the end and the start lie in $]x_1,x_2[$, 
\item either the start or the end lies in $]x_1,x_2[$ but not both, 
\item the start and end points are not in $]x_1,x_2[$.
\end{enumerate}
We illustrate in Figure \ref{3ed} the path $\bar c$ constructed in each case. 

{\it Case 1:} then $a \leq x_1 < 0 \leq z < x_2 \leq b$. Set $\ell_1=x_2-x_1\geq 2$, $\ell_2=b-x_2$ and $\ell_3=x_1-a$. The length of $\bar c_0$ is $\geq 3\ell_1+2\ell_2+2\ell_3+1$. The last $+1$ comes from the fact that there is at least one change of line.

We replace $\bar c_0$ by the following path $\bar c$. From the initial position one goes to the left until the $\Z$-coordinate is $a$, changes line, and moves to the right until the $\Z$-coordinate is $0$. Then one moves further to the right along zig-zag movement (\ie moves one step to the right then change line, and repeat these two moves over and over again) until the $\Z$-coordinate is $z$. Then move to the right until $b$, change line and move back to $z$. If required, change once more the line to be at the endpoint.

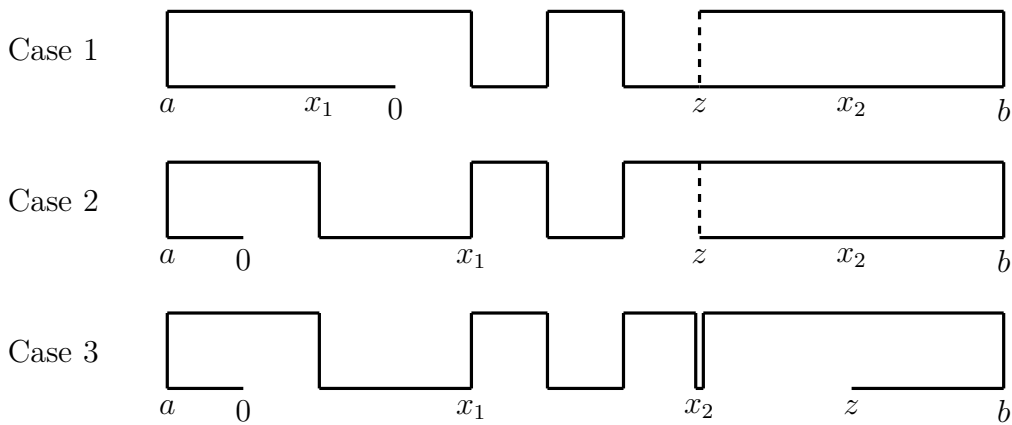
\begin{figure}
\begin{centering}
\begin{tikzpicture}
\draw (-4.5,0.5) node{Case 1};
\draw[very thick] (-3,0) -- (0,0) node[below]{$0$};
\draw (-1,0) node[below]{$x_1$};
\draw[very thick] (-3,1) -- (-3,0) node[below]{$a$};
\draw[very thick] (-3,1)--(1,1);
\draw[very thick] (1,1)--(1,0);
\draw[very thick] (1,0)--(2,0);
\draw[very thick] (2,0)--(2,1);
\draw[very thick] (2,1)--(3,1);
\draw[very thick] (3,1)--(3,0);
\draw[very thick] (3,0)--(4,0) node[below]{$z$};
\draw (6,0) node[below]{$x_2$};
\draw[very thick] (4,0) -- (8,0) node[below]{$b$};
\draw[very thick] (8,0)--(8,1);
\draw[very thick] (8,1)--(4,1);
\draw[very thick,dashed] (4,1)--(4,0);

\draw (-4.5,-1.5) node{Case 2};
\draw[very thick] (-2,-2)node[below]{$0$}--(-3,-2)node[below]{$a$};
\draw[very thick] (-3,-2)--(-3,-1);
\draw[very thick] (-3,-1)--(-1,-1);
\draw[very thick] (-1,-1)--(-1,-2);
\draw[very thick] (-1,-2)--(1,-2)node[below]{$x_1$};
\draw[very thick] (1,-2)--(1,-1);
\draw[very thick] (1,-1)--(2,-1);
\draw[very thick] (2,-1)--(2,-2);
\draw[very thick] (2,-2)--(3,-2);
\draw[very thick] (3,-2)--(3,-1);
\draw[very thick] (3,-1)--(8,-1);
\draw[very thick] (8,-1)--(8,-2)node[below]{$b$};
\draw[very thick] (8,-2)--(4,-2)node[below]{$z$};
\draw[very thick,dashed] (4,-1)--(4,-2);
\draw (6,-2)node[below]{$x_2$};

\draw (-4.5,-3.5) node{Case 3};
\draw[very thick] (-2,-4)node[below]{$0$}--(-3,-4)node[below]{$a$};
\draw[very thick] (-3,-4)--(-3,-3);
\draw[very thick] (-3,-3)--(-1,-3);
\draw[very thick] (-1,-3)--(-1,-4);
\draw[very thick] (-1,-4)--(1,-4)node[below]{$x_1$};
\draw[very thick] (1,-4)--(1,-3);
\draw[very thick] (1,-3)--(2,-3);
\draw[very thick] (2,-3)--(2,-4);
\draw[very thick] (2,-4)--(3,-4);
\draw[very thick] (3,-4)--(3,-3);
\draw[very thick] (3,-3)--(3.95,-3);
\draw[very thick] (3.95,-3)--(3.95,-4);
\draw[very thick] (3.95,-4)--(4.05,-4);
\draw (4,-4) node[below]{$x_2$};
\draw[very thick] (4.05,-4)--(4.05,-3);
\draw[very thick] (4.05,-3)--(8,-3);
\draw[very thick] (8,-3)--(8,-4)node[below]{$b$};
\draw[very thick] (8,-4)--(6,-4)node[below]{$z$};

\end{tikzpicture}
\end{centering}
\caption{Illustrations of the proof of Lemma \ref{3edge}.}\label{3ed}
\end{figure}

This path $\bar c$ visits all the vertices whose $\Z$-coordinate belong to $[a,b]$ so the lamp state can be set just like they are using the path $c_0$. Its length is $2(b-a)+2$ plus possibly $1$ for the change of line. As $\ell_1\geq 2$, this is strictly shorter than $\bar c_0$.

{\it Case 2:} By time reversal symmetry we may assume that the endpoint  $z \in ]x_1,x_2[$ but the  start-point $0 \leq x_1$. Thus we are lead to consider $a \leq 0 \leq x_1<z<x_2 \leq b$. 

By maximality of the intervalle $[x_1,x_2]$the edge between $x_1-1$ and $x_1$ is crossed only once in the increasing direction, say at time $t_1$. We call $c_1$ the restriction of the path $c_0$ to the time intervalle $[0,t_1]$. Necessarily, the length of $c_0$ is $\geq \ell(c_1)+3(x_2-x_1)+2(b-x_2)+1$, the final $+1$ corresponding to a change of line.

We replace the path $c_0$ by a path $c$ that coincides with $c_0$ until time $t_1$, then goes in zig-zag from $x_1$ to $z$ and finally makes a U-shape move from $z$ to $b$ and back as in Figure \ref{3ed}. The length of $c$ is equal to $\ell(c_1)+2(b-x_1)+1$ plus possibly $1$ for an final change of line. As $x_2-x_1 \geq 2$, this is strictly shorter than $c_0$.

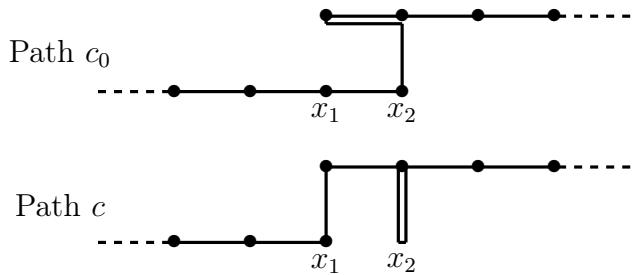
\begin{figure}
\begin{centering}
\begin{tikzpicture}
\draw (-1.5,0.5) node{Path $c$};
\draw[very thick, dashed] (-1,0)--(0,0);
\draw[very thick] (0,0)node{$\bullet$}--(1,0)node{$\bullet$};
\draw[very thick] (2,0)node[below]{$x_1$}--(1,0);
\draw[very thick] (2,0)node{$\bullet$}--(2,1)node{$\bullet$};
\draw[very thick] (2,1)--(2.95,1);
\draw (3,1)node{$\bullet$};
\draw[very thick] (2.95,1)--(2.95,0);
\draw[very thick] (2.95,0)--(3.05,0);
\draw[very thick] (3.05,1)--(3.05,0);
\draw (3,0)node[below]{$x_2$};
\draw[very thick] (3.05,1)--(4,1)node{$\bullet$};
\draw[very thick] (4,1)--(5,1)node{$\bullet$};
\draw[very thick, dashed] (6,1)--(5,1);

\draw (-1.5,2.5) node{Path $c_0$};
\draw[very thick, dashed] (-1,2)--(0,2);
\draw[very thick] (0,2)node{$\bullet$}--(1,2)node{$\bullet$};
\draw[very thick] (2,2)node[below]{$x_1$}--(1,2);
\draw[very thick] (2,2)node{$\bullet$}--(3,2)node{$\bullet$};
\draw (3,3) node{$\bullet$};
\draw[very thick] (3,2.9)--(3,2)node[below]{$x_2$};
\draw[very thick] (3,2.9)--(2,2.9);
\draw[very thick] (2,3)node{$\bullet$}--(2,2.9);
\draw[very thick] (2,3)--(3,3);
\draw[very thick] (3,3)--(4,3);
\draw[very thick] (4,3)node{$\bullet$}--(5,3)node{$\bullet$};
\draw[very thick, dashed] (6,3)--(5,3);

\end{tikzpicture}
\end{centering}
\caption{Illustration of the proof of Lemma \ref{3edge} : equality in Case 3. Note that the edge crossed twice over $x_2$ in the path $c$ can be replaced by an extra zag whenever $x_2<b$.}\label{3ed2}
\end{figure}

{\it Case 3:} In the last case, we have $a\leq 0 \leq x_1 <x_2 \leq z\leq b$. 

By maximality the edge between $x_1-1$ and $x_1$ is crossed only once in the increasing direction at a time $t_1$ and we call $c_1$ the restriction of $c_0$ to $[0,t_1]$ (if $a=x_1$, this edge is never crossed and $c_1$ is empty), and the edge between $x_2$ and $x_2+1$ is crossed only once   at a time $t_2$ and we call $c_2$ the restriction of $c_0$ to times $\geq t_2$ (if $b=x_2$, this edge is never crossed and $c_2$ is empty). The length of $c_0$ is $\geq \ell(c_1)+\ell(c_2)+3(x_2-x_1)+1$, where the last $+1$ is due to a change of line.

We replace $c_0$ between times $t_1$ and $t_2$ by a zig-zag path as in Figure \ref{3ed} to get a new path  $c$ of length $\leq \ell(c_1)+\ell(c_2)+2(x_2-x_1)+2$. As $x_2-x_1 \geq 1$, this new path is at most as long as $c_0$. There is equality if and only if $x_1+1=x_2$ (see Figure \ref{3ed2}).

In any case, we can replace $c_0$ by a no-longer path $c$ with no edge of $\Z$ used three times. Moreover, after applyng this procedures to all maximal intervalles $[x_1,x_2]$ with edges covered at least three times, we obtain a path $c$ such that $\bar c$ covers exactly twice the edges in $[a,0]\cup [z,b]$ and once those in $[0,z]$. Thus $\bar c$ is a solution to the projected travelling salesmen problem on $\Z$. 
\end{proof}

\begin{lemma}\label{obs1} Let $g,g' \in (\Z\times \Z_2)\wr L$.  We assume that the elements $g$ and $g'$ have the same projected paths $\bar{c} = \bar{c}'$. With the notations above, this implies in particular that $a=a'$, $b=b'$ and $z=z'$. We assume moreover that their lamp configurations are identical on vertices projecting to $]0,z[$, i.e. $f(x,\e)=f'(x,\e)$ for any $x \in ]0,z[$ and $\e \in \Z_2$.
 
Then the word length of $g$ and $g'$ differ by at most $4$.
\end{lemma}

\begin{proof} By symmetry assume $z \geq0$
Let $c$ be a path obtained by Lemma \ref{3edge} for $g$. When $z>0$, the path $c$ visits exacty one edge $\left((0,\e_1),(1,\e_1)\right)$ with $\e_1\in \Z_2$ and exacty one edge $\left((z-1,\e_2),(z,\e_2)\right)$ with $\e_2\in \Z_2$ (and they co\"{i}ncide when $z=1$).

Consider the path $c''$ defined by piecewise straight (possibly empty) lines from $(0,0)$ to $(a,0)$ to $(a,1)$ to $(0,1)$ to $(0,\e_1)$ then coincides with $c$  until $(z,\e_2)$ then piecewise straight to $(b,\e_2)$ to $(b,\e_2+1)$ to $(z,\e_2+1)$ to $(z,\e)$. This path goes from $0$ to $(z,\e)$ and allows to turn on appropritely the lamp configuration $f'$. Its length is the length of $c$ increased by at most $4$ (the vertical edges over $a,0,b,z$).

Thus $|g'| \leq |g|+4$. The same argument shows that $|g'|\leq |g|+2$ when $z=0$. The lemma follows by symmetry. 
\end{proof}

\begin{lemma}\label{obs2} Let $a_1 \leq a_2$ be two integers.
Let $F\subset \Z \times \Z_2$ denote the set of vertices with $\Z$-coordinates in $[a_1,a_2]$. The length of a travelling salesman path which visits all vertices of $F$, ends in $F$ and begins in $F$ is between $|F|-1$ and $|F|+1$.
\end{lemma}

\begin{proof}
A path which visits $k$ vertices has length at least $k-1$, which gives the lower bound. The upper bound is a simple (and slightly tedious) case by case argument. The solutions are paths of the form U-shape-zig-zag moves-U-shape as in Case 1 of Figure~\ref{3ed}.
\end{proof}

\begin{lemma}\label{banane}
Assume $g=(z,\e,f)\in (\Z\times \Z_2)\wr L$ is such that $z \geq 0$. For $\ell_0 \in L$, define the function $f'$ by 
\[
\forall \e' \in \Z_2, f'(x,\e')=\left\{\begin{array}{ll}
\ell_0 & \forall x \in [z,b], \\
f(x,\e') & \forall x\notin[z,b].
\end{array} \right.
\]

Then for any $(x,\e') \in [z,b]\times \Z_2$ there is a path $\{g_k\}_{k=0}^n$ such that $g_0=g$ and $g_n=(x,\e',f')$ and
\[
\forall 0\leq k\leq n, |g|-5 \leq |g_k| \leq |g|+3. 
\]

In particular, this is true for $(x,\e')=(z,\e)$ and for $(x,\e')=(b,0)$.
\end{lemma}

This lemma essentially asserts that starting from $g$, the lamp configuration on $[z,b]\times \Z_2$ can be turned to any chosen value (provided we keep a lamp on in $\{b\}\times \Z_2$) staying in an annulus of thickness $\leq 8$.

\begin{proof}
By Lemma \ref{3edge} there is a unique time $t_1$ where the path $c=c_0$ crosses the edge $(z-1,z)$ in the increasing direction, and we have 
\[
2(b-z) \leq \ell(c_{\{t\geq t_1\}}) \leq 2(b-z)+2.
\]
We construct a sequence of paths $c_k$ such that $\forall k \geq 1, c_{k|[0,t_1]}=c_{0|[0,t_1]}$ and $c_{k\{t\geq t_1\}}$ is a travelling salesmen problem visiting all vertices in $F=\pi^{-1}([z,b])$, starting in $c_0(t_1)$ and ending in $F$. By Lemma \ref{obs2}, we have $2(b-z)-1\leq\ell(c_{k\{t\geq t_1\}})\leq 2(b-z)+1$, so that 
\[
\ell(c_0)-3 \leq \ell(c_k) \leq \ell(c_0)+1.
\] 

Moreover, we can choose our sequence of paths $c_k$ to have as successive endpoints $E_k$ all the vertices in $F=\pi^{-1}([z,b])$ in some given order (following neighbours in the Cayley graph of $\Z \times \Z_2$). This enables us to construct successively an associated sequence of elements $\{g_k\}_{k=0}^n$ obtained by following the paths $c_k$ and switching lamps on appropriately according to $g_{k-1}$ on all points of $\Z \times \Z_2$ except at $E_k$ where the lamp is switched on to $\ell_0$. This yields eventually an element $g_m$ with lamp configuration $f_m=f'$. There remains to follow a similar path not changing the lamp configuration but bringing the endpoint to $(x,\e')$.

The argument in the proof of Lemma \ref{obs1} ensures
\[
\forall k \geq 1, \left| |g_k|-\ell(c_k)\right|\leq 2.
\]
Combining the two last inequalities gives the lemma.
\end{proof}

The previous lemma is already sufficient for the proof of bounded retreat depth, but let us first move on to the other main ingredient in the proof of Theorem \ref{techelle}.

\begin{lemma}\label{tank}
Assume $g=(z,\e,f) \in (\Z \times \Z_2)\wr L$ is such that $z\geq 0$. Then, there exists an integer $B\geq 0$ and a path $\{g_k\}_{k=0}^n$ such that $g_0=g$ and $g_n=(0,0,f')$, where the function $f'$ is defined for some $\ell_0 \in L\setminus \{e\}$ by
\[
\forall \e' \in \Z_2, f'(x,\e')=\left\{\begin{array}{ll}
\ell_0 & \forall x \in [0,B], \\
e & \forall x >B, \\
f(x,\e') & \forall x<0.
\end{array} \right.
\]
satisfying that 
\[
\forall 0 \leq k \leq n, |g|-5\leq |g_k| \leq|g|+5.
\]
\end{lemma}

This lemma asserts that starting from some $g$, we can replace the positive lamp configuration and the position $z$ by a configuration with fixed value $\ell_0$ and position $0$, staying in an annulus of thickness $\leq 10$. Of course a similar statement holds for the negative configuration by symmetry.

\begin{proof}
We construct by induction a sequence $\{h_j\}_{j=0}^z$ such that $h_j=(z-j,0,f_j)$ where 
\[
\forall \e' \in \Z_2, f_j(x,\e')=\left\{\begin{array}{ll}
\ell_0 & \forall x \in [z-j,B_j], \\
e & \forall x >B_j, \\
f(x,\e') & \forall x<z-j.
\end{array} \right.
\]
satisfies $|g|-5 \leq |h_j| \leq |g|+3$. The first term $h_0$ exists by Lemma \ref{banane}. To do the induction, assume $h_j$ is given. We construct $h_{j+1}$ as follows.

{\it Case 1:} If $|h_j|<|g|+3$, then move the lamplighter to $(z-j-1,0)$ turn the lamp on to $\ell_0$, move to $(z-j-1,1)$, turn on to $\ell_0$ move back to $(z-j-1,0)$. This is $h_{j+1}$ with $B_{j+1}=B_j$. The length may have been increased by $1$.

{\it Case 2:} If $|h_j|=|g|+3$. According to Lemma \ref{banane}, move the lamplighter to $(B_j,1)$, turn off the lamp there, move to $(B_j,0)$ turn off, move back to $(z-j,0)$. At this stage we reached an element $h'_j$ with $|h'_j|=|h_j|-2$. This new element qualifies for Case 1 with $B_j-1=B_{j+1}$. In the path connecting $h_j$ to $h'_j$, the length may be increased by at most $2$ according to Lemma \ref{obs2}.

For $j=z$, the element $h_z$ is our desired $g_n$. Observe that when we connect $h_j$ to $h_{j+1}$ the length may be increased by at most 1 in Case 1 and by at most 2 in Case 2.
\end{proof}
We are now ready to tackle the main result of this section.

\begin{proof}[Proof of Theorem \ref{techelle}]
Let us start with bounded retreat depth. Let $g=(z,\e,f)\in (\Z\times \Z_2)\wr L$ be an element of word length $n$. By symmetry, we may assume that $z \geq 0$.  By Lemma \ref{banane}, there is a path from $g$ to $g'=(b,\e,f')$ where $f'(x,\e')=e$ for all $x>b$ avoiding the ball of radius $n-6$. From $g'$, there is a path $\{g'_k\}_{\k \geq0}$ straightly connected to infinity given by $g'_k=(b+k,\e,f')$ of norm $|g'_k|=|g'|+k$. In particular, dead-ends may have at most retreat depth $5$.

Let $g \in S(n)^{\infty}$ be an element in the intersection of the sphere of radius $n$ and the infinite component of the complement of the ball of radius $n-1$. This implies that there is a path to some element $g_0$ of norm $|g_0|=m=n+12$. We claim that there exists two integers $A \leq 0 \leq B$ and a path from $g_0$ to $G=(0,0,F)$ where 
\[
\forall \e' \in \Z_2, F(x,\e')=\left\{\begin{array}{ll}
\ell_0 & \forall x \in [A,B], \\
e & \forall x \notin[A,B]. \\
\end{array} \right.
\]
Indead, we may assume $z\geq0$, apply Lemma \ref{tank} to reach $(0,0,f')$ and Lemma \ref{tank} again this time symmetrically in the case $z\leq0$. This gives a path from $g$ to $G$ in the annulus $S(m-10,m+10)$.

There remains to show that any two elements of the form $G$ for some $A$ and $B$ in $S(m-10,m+10)$ can be connected in $S(m-12,m+12)=S(n,n+24)$. This is done similarly to the proof of Lemma \ref{tank} using the key observation Lemma \ref{obs2}.
\end{proof}

In the proof of the boundedness of dead-ends, a thorough case checking indicates that the distance to identity does not decrease by more than $2$.
The sharp bound on dead-end depth seems to be $1$. Examples of such dead-ends are given by $(0,1,f_k)$ where 
\[
\forall \e' \in \Z_2, f_k(x,\e')=\left\{\begin{array}{ll}
\ell_0 & \forall x \in [-k,k], \\
e & \forall x \notin[-k,k]. \\
\end{array} \right.
\]
By looking at Lemma \ref{obs2}, one may check that one has to go through a element of length $|F|$ while the length of this element is $|F|+1$.

As for the connection thickness, the value is certainly much smaller than $24$. A careful look at the proof gives that $10$ is enough. The actual value is probably even lower.

\section{Lamplighter groups with infinite lamps}\label{sec:distortion}

In this section we consider a wreath product $\G \wr L$ with both $\G$ and $L$ infinite, and we assume that the group $L$ has no dead-end elements. We endow this group with the "switch or walk" generating set $S_\G \cup S_L$. The word distance is also related to a traveling salesman problem, by the 

\begin{proposition}\label{TSP2}
Let $g=(\g,f)\in \G \wr L$. The word norm for the generating set $S_\G \cup S_L$ is
\[
|g|=\sum_{x\in \G} |f(x)|_{S_L} +\textrm{length}(c),
\]
where $c$ is a solution to the travelling salesman problem in the Cayley graph of $(\G,S_\G)$ that starts in $e_\G$, visits all vertices in the support of $f$ and ends in $\g$.
\end{proposition}

\begin{proof}
Such a length is necessary because only the lamp at the position of the lighter is modified when we multiply by an element of $S_L$ and multiplying by an element of $S_\G$ moves the lighter by distance one only. Moreover given such a path, we can describe it by a word in $S_\G$ and we can insert the representative words of $f(x)$ in $S_L$ at the corresponding locations.
\end{proof}

\begin{theorem}\label{distortion}
Assume $\G$ and $L$ are infinite groups generated by finite sets $S_\G$ and $S_L$ respectively. Assume moreover that $(L,S_L)$ has no dead-end elements. Then in the group $\G \wr L$ with generating set $S_\G \cup S_L$, the connection thickness of spheres is $\leq 2$.
\end{theorem}

The idea of the proof is quite simple. We imagine the word norms of elements $f(x)$ are heigths of piles of bricks, and we consider the path solution to the associated travelling salesman problem. Then we move one by one the bricks at the endpoint of the path to the penultimate point. When this is done, we add an extra brick to allow us to shorten the path. By induction, we reduce the path to nothing and obtain a big pile of bricks at the origin. More precisely, here is a detailed proof.

\begin{proof}
Let $g=(\g,f)$ belong to $S(n,2)$, i.e. $n \leq |g| \leq n+2$. We can assume that $|g|=n+1$. Indeed, an element of length $n+2$ is always connected to one of length $n+1$, and if $|g|=n$, as $L$ has no dead end, we can multiply by an element in $S_L$ to increase the word norm of $f(\g)$ by one.

Now let $c=\{c(t)\}_{t=0}^T$ be a path solution to the travelling salesmen problem of Proposition \ref{TSP2}. In particular, $c(T)=\g$ and $c(0)=e_\G$.

We construct by induction a sequence $\{g_t\}_{t=T}^0$ such that $g_T=g$ and $\forall  0\leq t \leq T$, $g_{t-1}=\left(c(t-1),f_{t-1} \right)$ where
\[
f_{t-1}(x)=\left\{\begin{array}{ll}
f_t(x) & \textrm{if } x \notin \{c(t-1),c(t)\} \\
e_\G & \textrm{if } x =c(t), \\
\ell_{t-1} & \textrm{if } x =c(t-1), 
\end{array} \right.
\]
and $|\ell_{t-1}|=|f_t(c(t))|+|f_t(c(t-1))|+1$. We deduce from Proposition \ref{TSP2} that $\forall  0\leq t \leq T$, $|g_{t-1}|=|g_t|=n+1$ as at each step the length of the path decreases by one whereas the sum of word norms increases by one. Eventually $g_0=(e_\G, f_0)$ where the support of $f_0$ is reduced to the identity, and the value there $\ell_0$ has norm $n+1$.

We show that there is a path from $g_t$ to $g_{t-1}$ in the annulus $S(n,2)$. For this, we construct a sequence $\{h^t_s\}_{s=0}^S$ as follows. 

Let $w$ be  a reduced representative word for $f_t(c(t))$. We denote $w_s$ its prefix of length $|w|-s$. Moreover, as $L$ has no dead-ends, there exists a sequence $\{\l_s\}_{s\geq 0}$ in $L$ such that $\l_0=f(c(t-1))$ and $|\l_s|=|\l_0|+s$. We consider $h^t_s=(c(t),\f_s)$ where
\[
\f_s(x)=\left\{\begin{array}{ll}
f_t(x) & \textrm{if } x \notin \{c(t-1),c(t)\} \\
w_s & \textrm{if } x =c(t), \\
\l_s & \textrm{if } x =c(t-1), 
\end{array} \right.
\]
Clearly $|h^t_s|=n+1$ for all $t,s$. For $S=|w|$, we have $w_S=e_\G$. From $h^t_S$, move to $c(t-1)$ (the length decreases by one, as we delete the endpoint in a travelling salesman problem) and switch to $\l_{S+1}$ (the length increases by one). We reached $g_{t-1}$.

There remains to show that we can move from $h^t_s$ to $h^t_{s+1}$ in $S(n,2)$. While $s<S=|w|$, this is done as follows, depending on the relationship between the solutions of the travelling salesman problem for $g_t=(c(t),f_t)$ (the same as for $h^t_s$ for all $s$) and the modified problem starting in $e_\G$, visiting all points in $\textrm{supp}(f_t)\cup\{c(t)\}$ and ending in $c(t-1)$. As $c(t-1)$ and $c(t)$ are neighbours, the length of the solutions differ by at most one. We discuss the three possibilities.

{\it Case 1:} If the path modified is longer by one, switch the lamp at $c(t)$ to $w_{s+1}$ (length decreases by one), move to $c(t-1)$ (length increases by one), switch the lamp at $c(t-1)$ to $\l_{s+1}$ (length increases by one), move back to $c(t)$ (length decreases by one).

{\it Case 2:} If the path modified has the same length, switch the lamp at $c(t)$ to $w_{s+1}$ (length decreases by one), move to $c(t-1)$ (length remains equal), switch the lamp at $c(t-1)$ to $\l_{s+1}$ (length increases by one), move back to $c(t)$ (length remains equal).

{\it Case 3:} If the path modified is shorter by one, move to $c(t-1)$ (length decreases by one), switch the lamp at $c(t-1)$ to $\l_{s+1}$ (length increases by one), move back to $c(t)$ (length increases by one), switch the lamp at $c(t)$ to $w_{s+1}$ (length decreases by one).

This finishes the main part of the proof, connecting an arbitrary element of $S(n,2)$ to an element of the form $g_0=(e_\G,f_0)$ with $f_0$ supported on the identity taking value there $\ell_0$ of norm $n+1$. 

There remains to show that any two such elements are connected in $S(n,2)$. To do so, we fix an element $\ell_1 \in L$ of norm $n$ and an element $\g_1\in \G$ of norm $1$ and show that $g_0$ is connected to $(\g_1,f_1)$ where $f_1$ is supported on $\g_1$ taking value $\ell_1$ there. This is done as above, moving back and forth between $e_\G$ and $\g_1$ and switching to decreasing length prefixes of $\ell_0$ and increasing length prefixes of $\ell_1$.
\end{proof}

Note that the annulus $S(n,1)$ in the group $\Z \wr \Z$ with usual "switch or walk" generated set is not connected. 

\begin{proposition}[Theorem \ref{ZwrZ-intro}]
For the group $\Z\wr\Z$, endowed with the usual "switch or walk" generating set, there exists two constants $c_1,c_2>0$ such that the diameter of the graph $S(n,2)$ satisfies:
\[
\forall n \geq 0, c_1n^2 \leq \mathrm{diam}S(n,2) \leq c_2n^2.
\]
\end{proposition}

As follows from the proof below, the upper bound applies to any group $\G \wr L$ with $\G,L$ infinite and $L$ without dead-ends. We believe the lower bound is also valid in this more general context.

\begin{proof}
The upper bound is given by the proof of Theorem \ref{distortion}. Indeed, the path $\{h^t_s\}_{s=0}^S$ has length less than $4(|\ell_t|+1)$, where 
\[
|\ell_t|=\sum_{t'=t}^T |f(c(t'))|+T-t.
\]
As $T\leq n$ and $\sum_{t=0}^T |f(c(t))|\leq n$, we immediately deduce that the distance between $g$ and an element of type $g_0$ is less than $6n^2+\frac{n}{2}$. A similar distance connects $g_0$ to $(\g_1,f_1)$.

In order to prove the lower bound, we introduce the following notion.  The mass distribution associated to an element $g=(z,f)$ of $\Z \wr \Z$ is the sequence $\m(g)=(\m_x(g))_{x \geq 0}$ given by 
\[
\begin{array}{ll} 
\textrm{For } x \geq 1, & \m_x(g)= \mathds{1}_{[0,a]}(x)+\mathds{1}_{[z,a)}(x)+\left|f(x)\right|,\\
\textrm{For } x=0, & \m_x(g)= 2b-z_-+\sum_{x\leq 0}\left|f(x)\right|,
\end{array}
\]
where as in \S{}\ref{llline} we write $a=\max\left(\textrm{supp}(f)\cup\{0,z\} \right)$, $b=\min\left(\textrm{supp}(f)\cup\{0,z\} \right)$ and $z_-=\min\{z,0\}$. In other words, $\m_x(g)$ which we call the mass of $g$ at $x$ is the length of $f(x)$ plus the number of times the path solution to the associated travelling salesman problem visits $x$. By construction $|g|=\sum_{x\geq 0} \m_x(g)$.

Denote $g_{\textrm{end}}=(0,\d_0^{n+1})$ where $\d_0^{n+1}(x)=0$ for $n \neq 0$ and $\d_0^{n+1}(0)=n+1$ and $g_{\textrm{start}}=(n+1,Id)$, where $Id(x)=0$ for all $x \in \Z$. Both elements belong to $S(n+1)$ and we have $\m_x(g_{\textrm{start}})=(0,1,\dots,1,0,0\dots)$ and $\m_x(g_{\textrm{end}})=(n+1,0,\dots)$.

The graph $S(n,2)$ is bipartite between $S(n+1)$ and $S(n)\cup S(n+2)$ (this is inherited from $\Z$ and would not hold in arbitrary groups $\G$ and $L$). Therefore we may assume that at even times, a path between $g_{\textrm{start}}$ and $g_{\textrm{end}}$ takes values in $S(n+1)$. Let $g_t=(z_t,f_t)$ denote the $t^{\textrm{th}}$ even position of such a path. We observe that either $\m(g_t)$ and $\m(g_{t+1})$ co\"{i}ncide or the second is obtained from the first by lowering by one the mass at $x$ and increasing by one the mass at $x-1$ or at $x+1$. It follows that the length of a path in $S(n,2)$ between $g_{\textrm{start}}$ and $g_{\textrm{end}}$ has length at least $\frac{n^2}{2}$.
\end{proof}

\begin{remark}
One can check that for the group $(\Z \times \Z_2)\wr L$ considered in \S{}\ref{sec:ladder}, the diameter of $S(n,24)^\infty$ is also comparable with $n^2$.
\end{remark}

\begin{remark}\label{disthyp}
It is a fairly standard fact that, given two geodesic rays $\g_i: \N \to \Gamma$ in a hyperbolic group, then the length of the path (if it exists!) between $\gamma_1(t)$ and $\gamma_2(t)$ which avoids $B_{t-2}$ grows exponentially in $t$ (see, for example, \cite{Ger}). If the group is one-ended and finitely presented, this implies that most elements of $S(n,r)^\infty$ will be at distance $\geq K e^{Ln}$ for some $K,L>0$. On the other hand, since these spheres contain at most exponentially many elements, one has that the diameter is $\leq K' e^{L'n}$ (for some $K',L'>0$). By considering Cayley graphs which are triangulations of the hyperbolic plan, this is as sharp as possible: the sphere will be a cycle so the diameter is half the number of vertices.
\end{remark}

\section{General observations about spheres and dead-ends}\label{sgl}

\subsection{Direct products, spheres and retreat depth}\label{sdirprod}

If $G_1$ and $G_2$ are two finitely generated groups, then there are two ``natural'' generating sets for their direct product: $S_1 \perp S_2 := (S_1 \times \{e_2\}) \cup (\{e_1\} \times S_2 )$ and $S_1 \vee S_2 := (S_1 \cup \{e_1\} )\times (S_2 \cup \{e_2\} )$. The former will be referred to as the ``summed'' set and the latter as the ``product''.

With the summed generating set, one has $|(g_1,g_2)|_{S_\perp} = |g_1|_{S_1} + |g|_{S_2}$.  

With the product generating set, one has $|(g_1,g_2)|_{S_\vee} = \max( |g_1|_{S_1}, |g|_{S_2})$

For example, if $G_1=G_2 = \Z$ and $S_1 = S_2 = \{\pm 1\}$ then the summed generating set is the usual generating set of $\Z^2$ whereas the product generating set gives the ``king's move'' generating set.

\begin{lemma}\label{tconnprodsomm-l}
Assume $G_1$ and $G_2$ are infinite finitely generated groups. Then, for the summed generating set, $S_{G_1 \times G_2}(n,1)^\infty$ is connected.
\end{lemma}
\begin{proof}
First, note that elements of the form $(e_{G_1}, h_2)$ or $(h_1, e_{G_2})$ (where $|h_i|_{S_i} \in [n,n+1]$) are all connected together. Let $(g_1,g_2)$ and be an element of $S_{G_1 \times G_2}(n,1)^\infty$. We will show it is connected to an element of the former type (call these ``canonical elements'').

To do so, fix for each element $g_i \in G_i$ a geodesic $p_{g_i}$ from $e_{G_i}$ to $g_i$. By assumption there is a path from $(g_1,g_2)$ to infinity, say $(g_1^{(k)},g_2^{(k)})_{k=0}^\infty$. We will use this path to define another path $(h_1^{(k)},h_2^{(k)})_{k=1}^N$ which ends in a ``canonical'' element. For convenience, we will only produce a sequence $(h_1^{(k)},h_2^{(k)})_{k=1}^N$ of elements so that there is a path in $S_{G_1 \times G_2}(n,1)^\infty$ between two successive elements of the sequence.

The algorithm is as follows. Start at $(g_1,g_2) = (h_1^{(0)},h_2^{(0)})$. Each $h_i^{(k)}$ will belong to the geodesic $p_{g_i^{(k)}}$. 
Given $(h_1^{(k)},h_2^{(k)})$, go to $(h_1^{(k+1)},h_2^{(k+1)})$ as follows:
\begin{enumerate}
 \item find the $i$ such that $g_i^{(k)} \neq g_i^{(k+1)})$ (there is only one coordinate which changes). Without loss of generality, assume $i=1$.
 \item If $h_1^{(k)} = g_1^{(k)}$, go directly to (3). If $h_1^{(k)} \neq g_1^{(k)}$, slide the first coordinate ``up'' the geodesic $p_{g_1^{(k)}}$ while sliding the second coordinate ``down'' the geodesic $p_{g_2^{(k)}}$ (in order to keep the sum of the length of coordinates in $[n,n+1]$). Depending on whether $|g_1^{(k)}| \geq |g_1^{(k+1)})|$ or $|g_1^{(k)}| < |g_1^{(k+1)}|$, the total sum of the coordinates at the end of this step should be $n+1$ or $n$. If at any moment, it is no longer possible to slide the second coordinate down, then we are at a canonical element.
 \item Move the first coordinate to $g_1^{(k+1)}$. This is the element $(h_1^{(k+1)},h_2^{(k+1)})$.
\end{enumerate}
Note that it is impossible to continue this algorithm forever: the impossibility to slide down will occur at the latest when $\min( |g_1^{(k)}|, |g_2^{(k)}| ) \geq n+2$ (which is bound to happen since this is a path to infinity). In fact, by sliding up a coordinate and down the other, one can reach a canonical element as soon as $\max( |g_1^{(k)}|, |g_2^{(k)}| ) \geq n+1$.
\end{proof}

\begin{lemma}\label{tconnprodprod-l}
For the product generating set, $S_{G_1 \times G_2}(n,0) = S_{G_1}(n) \times B_{G_2}(n) \cup B_{G_1}(n) \times S_{G_2}(n)$. Furthermore, this is connected as soon as both $G_1$ and $G_2$ are infinite.
\end{lemma}
\begin{proof}
Recall that $(g_1,g_2) \in S_{G_1 \times G_2}(n,0)$ if and only if $\max \big( |g_i|_{S_i} \big) =n$. First, note, as in Lemma \ref{tconnprodsomm-l}, that elements of the form $(e_{G_1}, h_2)$ or $(h_1, e_{G_2})$ (where $|h_i|_{S_i} =n$) are all connected together. If $(g_1,g_2) \in S_{G_1 \times G_2}(n,0)$, then one of the $g_i$ is of length $n$. Thus, it is connected (inside $S_{G_1 \times G_2}(n,0)$) to one of the elements of the previous form by deleting the other factor.
\end{proof}

Recall that $\rd(G,S)$ is the retreat depth of $G$ for $S$. 
\begin{lemma}\label{tretdeprod-l}
Let $G_1$ and $G_2$ be infinite groups generated by the finite generating sets $S_1$ and $S_2$ respectively. Then
\[
\rd(G_1 \times G_2, S_1 \perp S_2) = \min\big(\rd(G_1,S_1), \rd(G_2,S_2) \big) \text{ and } \rd(G_1 \times G_2, S_1 \vee S_2) =0.
\]
\end{lemma}
\begin{proof}
The proof for the generating set $S_1 \vee S_2$ is actually contained in the proof of Lemma \ref{tconnprodprod-l}: the spheres of thickness $0$ are connected, since there is a geodesic going to infinity, any element $g$ can be moved inside $S_{|g|}$ until it intersects this geodesic and then go to infinity.

For $S_1 \perp S_2$, suppose, without loss of generality, that $\rd(G_1,S_1) \leq \rd(G_2,S_2)$. Let $g=(g_1,g_2)$ be a dead-end. By ignoring the second coordinate, one sees that the retreat depth of $g$ is at most that of its first coordinate. On the other hand, let $g_1$ be a dead-end of depth $\rd(G_1,S_1)$ and $g_2$ be a dead-end of depth $\rd(G_2,S_2)$. Let $g = (g_1,g_2)$. By looking at a path from $g$ to infinity, at least one of the projected path on a coordinate also goes to infinity. The depth achieved by the projected path is less than that of the original path. This shows the depth of $g$ is $\rd(G_1,S_1)$.
\end{proof}

This gives a very easy proof of
\begin{corollary}
Bounded retreat depth is not invariant under change of generating set.
\end{corollary}
\begin{proof}
Take $G_1 \times G_2$ where $G_1=G_2= \Z \wr \Z_2$ so that $\rd(G_1) = \rd(G_2) = + \infty$. By Lemma \ref{tretdeprod-l}, for $S_1 \perp S_2$ the retreat depth is infinite, whereas for $S_1 \vee S_2$ it is $0$.
\end{proof}
This corollary is also implied by \cite[Theorem 2.5]{War}, although we believe the above example to be much simpler. Also, if there is a finitely presented group with unbounded retreat depth, then one could answer Question (ii) [is bounded retreat depth independent of the generating set for finitely presented groups?] of \cite{Gou} in the negative. 

The same example gives:
\begin{corollary}
There is a group with connected spheres but unbounded retreat depth.
\end{corollary}
\begin{proof}
Take $G_1 \times G_2$ where $G_1=G_2= \Z \wr \Z_2$ with the generating set $S_1 \perp S_2$. By Lemma \ref{tretdeprod-l}, The retreat depth is infinite (since $\rd(G_1) = \rd(G_2) = + \infty$), but the spheres are connected by Lemma \ref{tconnprodsomm-l}.
\end{proof}

Note that $\rd(L) \neq 0$ if $L$ is finite. If $\Gamma$ is infinite and $\rd(L) =0$, then $\rd(\Gamma \wr L) =0$ as one can always go to infinity in the lamps. Actually, if $L$ is infinite, then $\rd(\Gamma \wr L) \leq \min( \rd(L), \rd(\Gamma) )$ by trying to go to infinity either in the lamps or in the base space. It seems likely that this inequality is sometimes strict.

Let us mention a remark about spheres in products of trees. They can be easily described and depend greatly on the chosen generating set. This is not surprising since we proved in \S{} \ref{sec:ladder} that spheres in a group can even be connected for a generating set and not-connected for another. However products of trees are much more elementary.

\begin{remark}
For $k \geq 2$, let $T_k$ be the free product of $k$ copies of $\Z_2$ endowed with the standard generating set (one letter for each factor). Its Cayley graph is then the infinite $k$-regular tree. Let $\Gamma = T_k \times T_\ell$. 
Let $S_\perp(n,1)$ be the spheres of radius $n$ and thickness $1$ for the summed generating set and $S_\vee(n)$ be the spheres of radius $n$ for the product generating set. These spheres are connected according to Lemmas \ref{tconnprodsomm-l} and \ref{tconnprodprod-l}. However, they tend to very different objects. 

The spheres $S_\perp(n,1)$ resemble greatly Diestel-Leader graphs, or more precisely their tetraedra in the sense of \cite{BW}. 
On the other hand, the spheres $S_\vee(n)$ can be described by the following picture. Consider $k^n \times \ell^n$ points on a horizontal plane in a rectangular grid shape (but with no edge in this planar picture) each of the $k^n$ rows of length $\ell^n$ are the leafs of an independent $\ell$-regular rooted tree drawn above the plane, and symmetrically, each of the $\ell^n$ columns of length $k^n$ are the leafs of an independent $k$-regular rooted tree drawn below the plane.The graph obtained by this collection of trees (connected by the points on the grid) is the sphere $S_\vee(n)$.

These two objects are very different. For instance one can prove that there is no uniform family of quasi-isometries $\phi_n : S_\perp(n,1) \to S_\vee(n)$.
\end{remark}

\begin{remark}\label{remquo}
Bounded retreat depth and connected sphere are not preserved under taking quotients. Indeed, let $G = \Z \wr \Z$. With the usual generators this group has no dead-ends and the spheres are connected. Now $\Z \wr \Z$ can be seen as $\Z[t,t^{-1}] \rtimes \Z$ where the automorphism is multiplication by $t$. The group $\Z \wr \Z_2 \simeq \Z_2[t,t^{-1}] \rtimes \Z$ (same automorphism) is a quotient of $G$, mapping usual generating set to usual generating set, and has neither connected spheres nor bounded retreat depth.
\end{remark}

\subsection{Rarity of dead-ends}\label{sec:rarde}

The aim here is to show that the dead-ends do not make an important part of the group.

There are some variations in the degree of not being a dead-end. We recall two subsets of the sphere of radius $n$ which are of interest:
\[
\begin{array}{lll}
S(n)^\infty &= \{ g \in S(n)  \mid \text{there is a path to } \infty \text{ avoiding } B(n-1) \} \\
S(n)^{s\infty} &= \{ g \in S(n)  \mid \text{there is a strictly increasing path to } \infty\} \\ 
\end{array}
\]
By a strictly increasing path from $g$ to $\infty$, we mean a 1-Lipschitz map $\pi:\N\rightarrow G$ with $|\pi(k)|=|g|+k$ for all $k$, see also Remark \ref{straightconnection}. Note that $S(n)^\infty  \supset S(n)^{s\infty}$.

\begin{lemma}
Assume $(G,S)$ has retreat depth bounded by $k$. Then,
\[
 |\cup_{i=1}^n S(n)^\infty | \geq \dfrac{1}{|S(k)|} |B(n)| - \frac{|B(k-1)|}{|S(k)|}.
\]
In other words, in any ball of radius $\geq k$, there is a positive fraction of elements whose retreat depth is $0$.
\end{lemma}
In particular, there is a positive fraction of elements which are not dead-end elements.
\begin{proof}
If there is a bound on the depth of dead-ends, then for any $x \in S(n)$ (with $n \geq k$) there is a $y \in S(n-k)^\infty$ so that $y$ lies on a geodesic between $e_G$ and $x$. This implies:
\[
|S(n)| \leq |S(n-k)^\infty| \cdot | S(k)| 
\]
Summing one gets
\[
|B(n)| \leq |B(k-1)| + \big( \sum_{i=k}^n |S(i-k)^\infty| \big) | S(k)|.
\]
Dividing by $|B(n)| \cdot |S(k)|$ and noting that $|B(n-k)| \leq |B(n)|$ gives
\[
 \frac{|B(n-k)|}{|S(k)|} \leq \frac{|B(k-1)|}{|S(k)|} + \big( \sum_{i=0}^{n-k} |S(i)^\infty| \big). \qedhere
\]
\end{proof}

Let us write down the following useful proposition from Funar, Giannoudovardi \& Otera \cite[Proposition 5]{FGO}:
\begin{proposition}\label{propfgo}
Let $\Gamma$ be a group and fix its Cayley graph for some finite $S$. Let $x \in G$ and let $C \subset G$ be a set of vertices so that $x$ is in a finite component of $\Gamma \setminus C$. Then $d(x,C) \leq \tfrac{1}{2} \diam C$.
\end{proposition}
\begin{proof}
Take any bi-infinite geodesic $p:\Z \to G$ going through $x$ with $p(0)=x$. Since $p$ is infinite and the component of $x$ finite, there are $n,m \in  \N$ such that $p(-n),p(m) \in C$. Hence $n+m \leq \diam C$. The conclusion follows by seeing that $d(x,C) \leq \min(m,n)$.
\end{proof}
Here is the corollary on retreat depth. 
\begin{corollary}\label{corrd}
The retreat depth of $x$ is $\leq \frac{|x|}{2}$ (\ie $\leq \lfloor \frac{|x|}{2} \rfloor$). 

In particular, 
\[
|S(2n)| \leq |S(n)^\infty| \cdot | S(n)| \text{ and } |S(2n+1)| \leq |S(n+1)^\infty| \cdot | S(n)|.
\]
\end{corollary}
\begin{proof}
Let $n = |x|$ (\ie $x \in S(n)$) and assume $x$ has retreat depth $k$. Apply \cite[Proposition 5]{FGO} (rewritten as Proposition \ref{propfgo} above) with $C = B(n-k)$: $x$ is in a finite component of $\Gamma \setminus C$. But $\diam C \leq 2n-2k$ and so $d(x,C)\leq n-k$. However $k = d(x,C)$ so $2k \leq n$.

For the inequality on spheres, note that, for any $x \in S(n)$, there is a $y \in S(\lceil n/2 \rceil)^\infty$ so that $y$ lies on a geodesic between $e_G$ and $x$.
\end{proof}
This seems to indicate that the dead-end elements do not make most of the group. The following lemma makes this a bit more precise for groups with exponential growth.

\begin{lemma}\label{lemcrpin}
Assume $G$ has exponential growth, \ie $|S(n)| \geq K \mathrm{exp}( L n)$ for some $K,L>0$. 
To any $n$ we associate a sequence $\{n_i\}_{i=0}^s$ with $s=\lfloor \log_2 n\rfloor+1$,  $n_s=\lceil\frac{n}{2}\rceil$ and $n_{i-1}=\lceil\frac{n_i}{2}\rceil$ (so that $n_0=1$).
Then, 
$\prod_{i=0}^n |S(n_i)^\infty|/ \mathrm{exp}(Ln_i) \geq K$. 

In particular, along any infinite sequence $\{n_i\}_{i \geq 0}$ satisfying $n_{i-1} =  \lceil n_i /2 \rceil$ and $n_0=1$, there are infinitely many $i$ so that $ |S(n_i)^\infty| \geq \mathrm{exp}(Ln_i)$.
\end{lemma}
\begin{proof}
Note that $\sum_{i=0}^{s} n_i = n$. 
Using Corollary \ref{corrd}, for any $n$,
\[
|S(n)| \leq \prod_{i=0}^{s} |S( n_i)^\infty|.
\]
Since $\prod_{i=0}^{s}  \mathrm{exp}(Ln_i) = \mathrm{exp}(Ln)$,
\[
|S(n)| \leq \mathrm{exp}(Ln) \prod_{i=0}^{s} |S(n_i)^\infty|/ \mathrm{exp}(Ln_i)
\]
or 
\[
1 \leq \frac{|S(n)|}{K \mathrm{exp}(Ln)} \leq  K^{-1}  \prod_{i=0}^{s} |S(n_i)^\infty|/ \mathrm{exp}(Ln_i).
\]
This implies $\limsup |S(n_i)^\infty|/ \mathrm{exp}(Ln_i) \geq 1$ and finishes the proof.
\end{proof} 

Say that a group has pinched exponential growth if there are constants $K,K'$ and $L$ such that for all $n$, $K \mathrm{exp}( L n) \leq |S(n)| \leq K' \mathrm{exp}( L n)$. Groups with pinched growth include solvable Baumslag-Solitar groups (see Collins, Edjvet \& Gill \cite{CEG}) and wreath products with base group $\Z$, see Johnson \cite{Joh} or Lemma \ref{sizeofsphere}. 

It is easy to see that this is not an invariant of generating set. Indeed, if $(G_1,S_1)$ has pinched exponential growth and $(G_2,S_2) = (\Z, \{\pm 1\})$, then $(G_1 \times G_2, S_\perp)$ has pinched exponential growth but $(G_1 \times G_2, S_\vee)$ does not. If instead we chose $(G_2,S_2) = (G_1,S_1)$, then $(G_1 \times G_2, S_\vee)$ has pinched exponential growth but $(G_1 \times G_2, S_\perp)$ does not.

Note that Lemma \ref{lemcrpin} implies that if $(G,S)$ has pinched exponential growth, then the sub-indices $i$ such that $|S(n_i)^\infty| \geq \tfrac{1}{K'} |S(n_i)|$ have a positive density. However this still leaves a startling gap between the result of Lemma \ref{lemcrpin} and Proposition \ref{relsize}.

Here are some further amusing general results. Let us start with another simple corollary of \cite[Proposition 5]{FGO} on the width of dead-ends. 
Recall that $\wid(x) := d\big( x, B_{|x|}^\comp \big)$ is the width of the dead-end element $x$.
Note that \emph{a priori} a trivial bound on $\wid(x)$ is $2|x|+1$ (\ie go to the identity and from there, go to some element of $S(|x|+1)$). 
\begin{corollary}
$\wid(x) \leq |x|+1$.
\end{corollary}
\begin{proof}
Let $n = |x|$ (\ie $x \in S(n)$) and assume $x$ is a dead-end element. Apply \cite[Proposition 5]{FGO} (rewritten as Proposition \ref{propfgo} above) with $C = B(n+1) \setminus B(n)$. Then $x$ is in the finite component $B(n)$ and $\diam C \leq 2n+2$ so $d(x,C) \leq n+1$.
\end{proof}

These sets have some sub-multiplicative properties:
\begin{lemma} 
\[
\begin{array}{c@{\,\leq\,}c@{\,\cdot\,}c}
|S(n+m)^{s\infty}| & |S(n)^{s\infty}| & | S(m)^{s\infty}| \\
\end{array}
\]
\end{lemma}
\begin{proof}
Indeed, assume  
$S(n+m)^{s\infty}$, then for any $y$ on a geodesic from $e$ to $x$, $y  \in S(|y|)^{s\infty}$. Indeed concatenate the strictly increasing path from $y$ to $x$ (taken from that geodesic from $e$ to $x$) with the  strictly increasing path from $x$ to infinity. Note that for any $z \in B(|y|)^\comp$, $d(e,z) = |y| + d( S(|y|),z)$. Let $\pi:\N \to G$ be a strictly increasing path from $x$ to infinity (with $\pi(0) = x$). We need to prove that this path is still strictly increasing from $y$. Since $d\big(y,\pi(n)\big) \geq d\big(S(|y|),\pi(n)\big)$ and the right-hand side increases as $n$ does, $n \mapsto d\big(y,\pi(n)\big)$ is also strictly increasing (because the only allowed variations are $\pm 1$ and $0$). [This argument essentially only uses that strictly increasing paths are always geodesic.] 

This shows that the number of $x \in S(n+m)$ which have a strictly increasing path to infinity are at most the number of $y \in S(n)^{s\infty}$ times $|S(m)^{s\infty}|$.
\end{proof}

\begin{remark}
\begin{enumerate}
\item Another funny inequality is that, since any element $x \in S(n-k)^{s\infty}$ has an element $y \in S(n)^{s\infty}$ so that $x$ lies on a geodesic from $e_G$ to $y$, for any $k$,
\[
|S(n-k)^{s\infty}| \leq |S(n)^{s\infty}| \cdot | S(k)|.
\]
This is \emph{a priori} not true if one replaces $S(i)^{s\infty}$ by $S(i)$.

\item Recall (see Kellerhals, Monod \& R{\o}rdam  \cite{KMR}) that a group is not supramenable when it contains a bi-Lipschitz embedding of a binary tree. This implies, that if a group is not supramenable then $S(n)^{s\infty}$ has exponential growth. In general, $|S(n)^{s\infty}| \geq 2$ because there is always a bi-infinite geodesic line in a Cayley graph. 
 
Using Wilkie \& van der Dries \cite[(1.10) Corollary]{WvdD}, if $|S(n)^{s\infty}|$ is not bounded then it is $\geq n$.

\item If $|S^{s\infty}(n)|$ is bounded, note that there are only finitely many weakly geodesic rays (in the sense of Webster \& Winchester \cite[Definition 1.1]{WW}). However, a result of Rieffel (see \cite[Theorem 1.1]{WW} and references therein) would then imply that the ``metric boundary'' of the Cayley graph is finite. Though there are no results known to the authors, this seems to be a property which only holds for virtually cyclic groups.
\end{enumerate}
\end{remark}

\newcommand{\sd}[0]{\mathrm{sd}}
Lastly, in order to mimic Corollary \ref{corrd} (and thus hope to obtain the corresponding inequality for $S(n)^{s\infty}$), it seems convenient to introduce the depth of the dead-end's ``shadow'', \ie
\[
\sd(x) = \min \{ k \in \N \mid \exists y \in S(|x|-k)^{s\infty} \text{ such that } y \text{ is on a geodesic from } e_G \text{ to } x\}.
\]
In other words, it tells us how far one has to go back before one can take a strictly increasing path to infinity. The depth of the shadow can be much larger than the retreat depth: indeed,
\[
\wid(x) \leq 2 \sd(x) +1. 
\]
[Compare with $2\rd(x)+1 \leq \wid(x)$.] Since there are groups with bounded retreat depth but unbounded width, the depth of the shadow can be unbounded even if the retreat depth of dead-ends is bounded.
\begin{question}
Is $\sd(x) \leq \lceil |x|/2 \rceil$? 
\end{question}
A positive answer would allow to prove an analogue of Lemma \ref{lemcrpin} for $S(n)^{s\infty}$.

\section{Remarks on related works.}\label{other}

\subsection{Connected boundaries}\label{ssconnbnd}

Throughout this section, one may consider a graph $G$ instead of only a group. Let $d$ denote the graph distance. For a subset $K$ of vertices of the Cayley graph, let the $r$-neighbourhood be 
\[
N_r(K) =  \{ x \in G \mid d(x,K) \leq r \}.
\]
Let us define two notions of ``connected boundaries'' which is the natural extension of ``connected spheres''.
\begin{definition}
A set $F$ is said $\comp$-connected if $F$ is connected, its complement $F^\comp$ is connected and, for any end $\xi$, one cannot reach $\xi$ both from $F$ and $F^\comp$.
\end{definition}
As an example, note that in a one-ended graph, if $F$ is $\comp$-connected then only one of $F$ and $F^\comp$ is infinite. \emph{A contrario}, on a tree with strictly more than one end, if $F$ is $\comp$-connected then both $F$ and $F^\comp$ must be infinite.
\begin{definition}
Assume $r \in \N$, a graph has $CB_r$ if for any $\comp$-connected set $F$ one has that $N_{r+1}(F) \setminus F$ is connected. A graph has $CB'_r$ if for any $\comp$-connected set $F$, one has that $N_r (F) \cap N_r(F^\comp)$ is connected.
\end{definition}

Theorem \ref{casfinpres} can easily be extended (as is the aim and underlying result of \cite{Timar}) to show that this is true in all groups with finite presentation. In fact, if words have length at most $R$ then the group has $CB_{\lfloor R/2 \rfloor}$ and $CB'_{\lceil R/4 \rceil}$. 

Lastly let us recall the constant $C_G$ from \cite{BB}. Denote by $\hat{G}$ the end compactification of $G$. Call a cutset a set of edges\footnote{Normally, one allows edges and vertices. Allowing edges and vertices does not decrease the value of $C_G$ but may increase it by at most $2$.} $C$ such that $\hat{G} \setminus C$ is not connected. It is said minimal, if no strict subset of $C$ is also a cutset. Sets $\comp$-connected sets and minimal cutsets are in $2$ to $1$ correspondence: the boundary of a $\comp$-connected set is a minimal cutset; if $C$ is a minimal cutset, then $\hat{G}\setminus C$ has two connected components each of which is a $\comp$-connected set.

Say a cutset is $l$-close if for any partition $A_1 \sqcup A_2 = C$, $d(A_1,A_2) \leq l$. Then
\[
C_G = \sup \{ l \mid \text{there exists } C \text{ a cutset which is } l-\text{close} \}.
\]
Let us show that all these notions are equivalent.
\begin{lemma}\label{tequivbordconn-l}
Let $G$ be any graph.
\begin{enumerate} \renewcommand{\labelenumi}{ \rm (\alph{enumi})} 
\item $G$ has $CB_r \implies C_G \leq 2r+1$.
\item $C_G \leq r \implies G$ has $CB'_{\lceil r/2 \rceil}$.
\item $G$ has $CB'_r \implies G $ has $CB_{2r}$.
\end{enumerate}
\end{lemma}

If these conditions hold for some $r$ in a graph $G$, we say that $G$ has connected boundaries. It obviously implies that $G$ also has connected spheres.

\begin{proof}
(a) Assume $C$ is a cutset and $A_1 \sqcup A_2$ a partition. Let $F$ be a $\comp$-connected set associated to $C$. By hypothesis $\del_rF := N_{r+1}(F) \setminus F$ is connected. For $i=1$ and $2$, pick $a_i$ a vertex which is incident with an edge of $A_i$. Let $\gamma: \N \cap [0,k] \to \del_r F$ be a path so that $\gamma(0) = a_1$ and $\gamma(k) = a_2$. Note that for any $t$ one has either $d(A_1,\gamma(t)) \leq r$ or $d(A_2,\gamma(t)) \leq r$. The first condition holds for $t=0$. If the second also does then $C_G \leq r$. Otherwise, let $t_0$ be the first $t$ where the first condition fails. Then $d(A_1,\gamma(t_0)) = r+1$ and $d(A_2,\gamma(t)) \leq r$. This implies $d(A_1,A_2) \leq 2r+1$.

(b) Assume $F$ is a $\comp$-connected set and let $C$ be the edges at its boundary. Note that $\del'_r F = N_r(F) \cap N_r(F^\comp) = N_{r-1}(C)$. Let $c = \lceil C_G/2 \rceil$. Assume that $\del_c' F$ is not connected. Then there is partition $X_1 \sqcup X_2 = \del'_c F$ so that $X_1$ and $X_2$ are not connected (each $X_i$ may have many connected components). Let $A_i = C \cap X_i$. By hypothesis, $d(A_1,A_2) \leq C_G$. Hence there is a path of length $\leq C_G$ from some vertex incident with $A_1$ to some vertex incident with $A_2$. This path is however contained in $\del'_c F$, a contradiction.

(c) Take a $\comp$-connected set $F$. We want to show that $\del_{2r} F := N_{2r+1}(F) \setminus F$ is connected. To do so notice that $N_{r}(F)$ is a connected set. Let $Y_i$ be the connected components of $N_r(F)^\comp$. Each $Y_i$ is a $\comp$-connected set, hence $\del_r' Y_i = N_r(Y_i) \cap N_r(Y_i^\comp)$ is connected. Also $\del_r' Y_i  \subset \del_{2r} F$. Because $F$ is $\comp$-connected, for any $x,x' \in \del_{2r} F$ there is a path from $x$ to $x'$ which stays in $F^\comp$ (because $F$ is $\comp$-connected). 

Assume this path $\gamma$ leaves $N_{2r+1}(F)$. Then there is a $i$ so that it leaves through $\del'_r Y_i$. If it enters through $\del'_r Y_j$, note that $Y_j$ and $Y_i$ are in the same connected component. Hence, $i=j$, or in words: the path enters through $\del'_r Y_i$ again. Let $t_0$ be the smallest $t$ so that $\gamma(t+1) \neq N_{2r+1}(F)$ and $t_1$ be the smallest $t$ so that $\gamma(t_1-1) \neq N_{2r+1}(F)$ but $\gamma(t_1) \in N_{2r}(F)$. Then, since the $\del'_r Y_i$ are connected, there is a path, lying inside $\del_r'Y_i$ (a subset of $\del_{2r}(F)$) between $\gamma(t_0)$ and $\gamma(t_1)$. Hence one can modify $\gamma$ so that it lies in $\del_{2r}F$.
\end{proof}
Note that by concatenation the previous lemma gives $CB_r \implies CB'_{r+1}$. This is a direct consequence that $\del_r F \subset \del'_r F$ and any vertex in $N_r(F^\comp) \cap F$ is connected to some element of $N_1(F)$.

Also note that the hypothesis on ends in the definition of $\comp$-connected and of $C_G$ does not play any role in the proof of Lemma \ref{tequivbordconn-l}.

\begin{remark} There is an unfortunate gap in the proof of \cite[Theorem 3.2]{Timar}; hence, it is not clear whether the property of connected boundaries is invariant under change of generating sets, bi-Lipschitz equivalences or quasi-isometries. 
\end{remark}

We do however note that the lamplighter on the ladder does not produce a counter-example to the invariance under change of generating set:
\begin{lemma}
The property of connected boundaries does not hold for the same group and generating set as in Theorem \ref{techelle}.
\end{lemma}
\begin{proof}
The idea is essentially as in \cite{Timar}. Consider the $F_n$ whose elements correspond to any lamp state on the vertices with $\Z$ coordinates in $[-n,n] \subset \Z$ and the walker having the $\Z$ coordinate also in $[-n,n]$. The edges in $\del F_n$ have an incident where the lamplighter has $\Z$ coordinate in absolute value equal to $n+1$. Consider the obvious partition of this set into left (the $\Z$ coordinate is $-n-1$) and right (where it is $n+1$). Any path between those two sets has length at least $2n$.
\end{proof}

\subsection{Almost-convexity}\label{ssalmcvx}

\begin{definition}
Let $r \in \Z_{\geq 1}$. A group is said to be $r$-almost-convex if there is an integer $N_r$ so that for any $g_1,g_2 \in G$ such that $d(g_1,g_2) \leq r$ there is a path of length $\leq N_r$ between $g_1$ and $g_2$ which, except for the last and first edge, lies inside $B(m-1)$ for $m= \min( |g_1|, |g_2|)$. 

A group is almost convex if it is $r$-almost-convex for every $r \in \Z_{\geq 1}$.
\end{definition}
Cannon \cite[Theorem 1.3]{Cannon} shows $2$-almost-convex implies $r$-almost-convex with $N_r \leq r N_2^{ 1+ r/2 }$. Obviously $\{N_i\}_{i=1}^\infty$ form an increasing sequence (if one admits $\infty$ as a value). Hence Cannon shows that $2$-almost-convex 
is equivalent to almost convex. This property depends on the generating set.

Thiel \cite{Thiel} showed that this property may depend on the generating set even for nilpotent groups. By \cite[Theorem 1.4]{Cannon} almost-convex groups are finitely presented, hence have connected spheres and connected boundaries for all generating sets. 

Almost-convexity, connected spheres and bounded retreat depth of dead-ends seem to address different aspects of the word metric. Connected spheres cannot imply almost convexity because, although it gives the existence of a path, there is no bound on its length. As is implicit in Lemma \ref{tequivbordconn-l} and Lemma \ref{monot}, connected spheres is only concerned about the possibility that there is a partition $A_1 \sqcup A_2$ of the $n_i^\text{th}$ sphere so that $d(A_1,A_2) \geq i$.

The lamplighter on the ladder or $\Z \wr \Z$ show that there are groups which are not almost-convex but have connected spheres and bounded retreat depth of dead-ends. The solvable Baumslag-Solitar groups (see \cite{MS}) are also not almost-convex and have connected spheres, being finitely presented. More examples can be constructed by using Lemmas \ref{tconnprodsomm-l}, \ref{tconnprodprod-l} and \ref{tretdeprod-l} as well as noting that $(G_1,S_1)$ and $(G_2,S_2)$ are almost convex if and only if $(G_1 \times G_2, S_1 \vee S_2)$ is almost-convex if and only if $(G_1 \times G_2, S_1 \perp S_2)$ is almost-convex. 

It is easy to deduce from \cite[Lemma 3]{Bog} that hyperbolic groups have a bound on the width (hence retreat depth) of dead-ends. More generally, a group with regular language of geodesics has dead-ends of bounded width \cite[Theorem 1.1]{War-deep}. Further, if the group is Abelian, there are only finitely many dead-ends \cite[Theorem 1]{Lehnert}. Any group with more than one end has a bound on the width of dead-ends by \cite[Theorem 2]{Lehnert}. 

L.~Ciobanu pointed to the second author that putting together \cite{War-deep} and \cite{War-Heis}, one sees the Heisenberg group does not have a regular language of geodesics for any generating set.

\section{Questions.}\label{qu}

We finally list some open questions. Some of them are folklore and some arose in the course of this paper. Their difficulty is probably diverse. In general, it would be interesting to understand spheres in more Cayley graphs.

\begin{question}\label{qnatur}
Is there a finitely generated one-ended group $G$ which, for some generating set $S$, does not have connected spheres and admits a bound on the retreat depth of dead-ends~?
\end{question}
This should be compared with the first statement of Proposition \ref{prop-simple-intro}. Note that dead-end elements are actually quite frequent in groups so that an answer to Question \ref{qnatur} with no dead-end elements at all seems daunting to construct.

According to \cite{War} and Theorem \ref{teo-echelle-intro}, bounded retreat depth and connected spheres are dependent on the generating set. It seems natural to ask whether one can always choose a generating set with bounded retreat depth and connected spheres.
\begin{question}
Is there a group which never has connected spheres~? never has bounded retreat depth~?
\end{question}
A likely candidate to answer both questions is $T_k \wr L$ with $k \geq 3$ and $L$ finite studied in~\S{}\ref{slltree}.

The ``universal'' bound on retreat depth from \cite{FGO} (see Corollary \ref{corrd}) is still far from known examples. Indeed, in the lamplighter on the line (or in Houghton's group $H_2$, see \cite{Lehnert}), the bound on retreat depth is $\rd(x) \leq |x|/4$. This raises the
\begin{question}
Can the ratio $\frac{\rd(x)}{|x|}$ be arbitrarily close to $\frac{1}{2}$~?
\end{question}

As pointed out in Remark \ref{remquo}, bounded retreat depth and connected spheres are not preserved under taking quotient, as for instance $\Z \wr \Z$ has both properties but its quotient $\Z \wr \Z_2$ has none. 

It would be interesting to know what happens to the following other quotients. 
Let $p,q \in \N$ be coprime and consider $G_{p,q} := \Z[\tfrac{p}{q},\tfrac{q}{p}] \rtimes \Z$ where $\Z$ acts on $\Z[\tfrac{p}{q},\tfrac{q}{p}]$ by multiplication by $p/q$.
The group $G_{p,q}$ is the metabelianisation of the Baumslag-Solitar group $BS(p,q)$ and in particular, $BS(1,q) \cong G_{1,q}$.
However, if $p$ and $q$ are coprime and $p\neq 1 \neq q$ then $G_{p,q}$ is not finitely presented, see \cite[Theorem.(iii) in \S{}1.4 on p.48]{BS}.

\begin{question}
Is it true that spheres are connected and the retreat depth is bounded in $G_{p,q}$ for the ``switch-walk-switch'' or ``walk or switch'' generating sets inherited from $\Z \wr \Z$~?
\end{question}

Note that the group $BS(1,m)$ has connected sphere (being finitely presented) but it is not known whether it has bounded retreat depth or not. For the generating set induced from ``walk or switch'', \cite[Theorem 2.3]{War} shows that the width of dead-ends is arbitrarily large. 

This leads us to wonder what happens for extensions. First, note that if $1 \to K \to G \to Q \to 1$ the extension $G$ may still be nice even though $Q$ or $K$ have unbounded retreat-depth and non-connected spheres. For instance a direct product $G = Q \times K$ always admits a generating set with connected spheres and bounded retreat depth - see Lemmas \ref{tconnprodsomm-l}, \ref{tconnprodprod-l} and \ref{tretdeprod-l}. However the following more precise question is still open. 
\begin{question}
Is there a group $G$ with finite generating set $S$ such that there is an exact sequence $1 \to K \to G \overset{\pi}{\to} Q \to 1$ where $Q$ has bounded retreat-depth and connected spheres with respect to the induced generating set $\pi(S)$ and $G$ has unbounded retreat-depth ? $G$ does not have connected spheres ? both~?
\end{question}
It could be that the above is not possible under the additional constraint that $K$ is finitely presented.
An interesting candidate is the third Houghton group $H_3$ as there is an exact sequence $1 \to \mathrm{Sym}_{fin} \to H_3 \to \Z^2 \to 1$ where $\mathrm{Sym}_{fin}$ is the set of permutations with finite support on some countable set. The arguments of \cite{Lehnert} for $H_2$ do not directly pass to $H_3$.

Theorem \ref{nonconnected-intro} provides some computations of connection thickness functions. It would be interesting to have more exemples of such functions.

\begin{question}\label{thgrowth}
What are the possible behaviors of connection thickness functions ? Are there one-ended groups with $\mathrm{th}_{\G,S}(n)$ growing faster than $n+2$ ?
\end{question}

The successive partitions $\Pi(n,r)$ and $\Pi(n,r)^{\infty}$ defined in \S{}\ref{entropy} associated to the annuli $S(n,r)$ for $0 \leq r \leq \mathrm{th}_{\G,S}(n)$ are refinement of one another when restricted to $S(n)$ or $S(n)^\infty$. The sphere is most disconnected by the partition $\Pi(n,0)$ and is connected in $\Pi(n,\mathrm{th}_{\G,S}(n))$. It would be interesting to understand how this connection phenomena occurs. This can be thought of as some kind of geometric percolation.

We introduced the normalised entropy of a partition in order to quantify this phenomena. Proposition \ref{entmax} essentially asserts that the disconnection of the partition $\Pi(n,r)$ in $(\Z \wr L, L\{\pm 1\}L)$ with $L$ finite is in a sense maximal while $0 \leq \frac{r}{\mathrm{th}_{\G,S}(n)} < \frac{1}{8}$. It would be interesting to understand what happens in the interval $[\frac{1}{8},1]$.

\begin{question}
How does the connection of spheres occur ? Is there emergence of a gigantic component ? If yes, around which value of $r$ ? How does the normalised entropy behave in terms of the limit of the ratio $\frac{r}{\mathrm{th}_{\G,S}(n)}$ in the intervalle $[0,1]$ ? 
\end{question}

An answer to these questions for the lamplighters on trees with finite lamps of \S{}\ref{slltree}  would already be interesting.

We pointed out in Remark \ref{straightconnection} the surprising fact that the proportion of points $x$ in the sphere of $\Z \wr \Z_2$ not straightly connected to infinity (meaning there exists a geodesic ray $g(t)$ defined for all $t \geq 0$ with $|g(t)|=|x|+t$) is asymptotically positive. 
\begin{question}\label{qustra}
Are there exemples of finitely generated groups where the proportion of points straightly connected to infinity is arbitrarily close to $0$ ?
\end{question}

We were lead to consider the lamplighter group on a ladder in our study of connection of spheres in the lamplighter group $\Z^2 \wr \Z_2$ on the plane. We formulate the

\begin{conjecture}
In $\Z^2 \wr \Z_2$ for the usual ``switch-walk-switch'' generating set, the spheres are connected.
\end{conjecture}
We also believe the retreat depth of dead-ends is bounded. However, we point out that it is not even obvious to adapt \S{}\ref{sec:ladder} to show that any group $(\Z\times F)\wr \Z_2$ for $F$ finite with ``switch-walk-switch'' generating set has connected spheres.

Observe that by analogy with the case of the line \S{}\ref{llline} or the tree \S{}\ref{slltree}  it is tempting to expect that  dead-end elements with large retreat depths in the lamplighter over the plane $\Z^2$ could be ``all lamps state in a ball around the identity in a dead-end position and the lamplighter back at the identity''. But this does not always work: for example, take $\Z^2$ with ``king's move'' generators (\ie $\{(a,b) \mid a,b = 0,1 $ or $-1\}$. 
It is not difficult to see that for any two distinct vertices in the ball of radius $n$ (which look like squares of side $2n+1$), there is a Hamiltonian path covering the whole ball which joins these vertices. This shows that the retreat depth of such elements is $1$.

When a group has connected spheres, one can compare the metric on $S(n)^\infty$ given by the graph metric in the ambient Cayley graph with the graph metric induced by the graph $S(n,r)^\infty$. A natural generalisation of the study of connectedness is that of  distortion between these two metrics. We ask the
\begin{question}
Is there a Cayley graph for which $S(n,r)^\infty$ (or a subsequence) is a family of expander graphs ? Is there a Cayley graph for which the diameter of $S(n,r)^\infty$ is comparable to a linear function of $\log |S(n)|$ ?
\end{question}

A positive answer to the first question would also give a positive answer to the second. In Theorem \ref{ZwrZ-intro}, we obtain that the diameter of $S(n,2)$ for $\Z \wr \Z$ with ``switch or walk" generating set is comparable to $n^2 \asymp (\log|S(n,2)|)^2$ - see also Remark \ref{disthyp}. Of course, the diameter of $S(n)$ for an infinite group
is always $2n$ with respect to the ambient metric. It would also be interesting to compare with the random metrics on the sphere introduced by Georgakopoulos \cite{Geo}.

Finally, we recall here a classical question about amenability and spheres. It is an easy exercise to show that $G$ has subexponential growth if and only if a density one subsequence of the sequence of balls is F{\o}lner, \ie $\frac{|S(n_k+1)|}{|B(n_k)|}\rightarrow 0$ with $\frac{1}{N}|\{n_k\leq N\}|\rightarrow 1$. However the following remains open:

\begin{question}
Let $G$ have subexponential growth, is it true that $\frac{|S(n+1)|}{|B(n)|}\rightarrow 0$ ?
\end{question}

This is known to be true for $G$ virtually nilpotent by Pansu \cite{Pan}, using the fact that the asymptotic cone is a nilpotent Lie group with the Carnot-Caratheodory metric.

Conversely, it is easily checked that if $G$ has exponential growth, a subsequence of the sequence of balls cannot be F{\o}lner unless it has density $0$, i.e. $\frac{|S(n_k+1)|}{|B(n_k)|}\rightarrow 0$ implies $\frac{|\{n_k\leq N\}}{N}\rightarrow 0$. This is the content of the proof of Lemma 2.2 in \cite{Pit}. The statement of Pittet's Lemma 2.2 is not proved and is still open :

\begin{question}
Let $G$ have exponential growth, does there exist $\e>0$ such that  $\frac{|S(n+1)|}{|B(n)|}\geq \e$ for all $n$ ?
\end{question}

Note that the answer is positive for groups of pinched exponential growth.

\textsc{\newline J\'er\'emie Brieussel} \newline Universit\'e de Montpellier 
\newline Institut Montpelli\'erain Alexander Grothendieck \newline 34095 Montpellier, France \newline
 jeremie.brieussel@umontpellier.fr

\textsc{\newline Antoine Gournay }
\newline Institut f\"ur Geometrie, Fachrichtung Mathematik \newline TU Dresden \newline 01062 Dresden, Germany \newline 
antoine.gournay@tu-dresden.de

\end{document}